\documentclass[a4paper]{amsart}
\usepackage[nobysame]{amsrefs}
\usepackage{hyperref}

\newtheorem{theorem}{Theorem}[section]
\newtheorem{lemma}[theorem]{Lemma}
\newtheorem{prop}[theorem]{Proposition}

\newcommand{\D}{{\rm d}}

\newcommand{\R}{{\mathbb R}}

\newcommand{\N}{{\mathbb N}}

\newcommand{\E}{\mathbb{E}}
\newcommand{\PP}{\mathbb{P}}

\newcommand{\HH}{{\mathcal H}}
\newcommand{\ii}{\mathbf{1}}

\DeclareMathOperator{\Var}{Var}
\DeclareMathOperator{\inti}{int}

\numberwithin{equation}{section}

\title[Random polytopes circumscribed around a convex body]{The volume of random polytopes circumscribed around a convex body}
\author{Ferenc Fodor}
\address{Department of Geometry, Bolyai Institute, University of Szeged, Aradi v\'ertan\'uk tere 1, H-6720 Szeged, Hungary, and Department of Mathematics and Statistics, University of Calgary, Canada}
\email{fodorf@math.u-szeged.hu}
\author{Daniel Hug}
\address{Karlsruhe Institute of Technology, Department of Mathematics,
D-76128 Karls\-ruhe, Germany}
\email{daniel.hug@kit.edu}
\author{Ines Ziebarth}
\address{Karlsruhe Institute of Technology, Department of Mathematics,
D-76128 Karls\-ruhe, Germany}
\email{ines.ziebarth@kit.edu}

\subjclass[2010]{Primary 52A22, Secondary 60D05, 52A27}

\begin{document}

\begin{abstract}
Let $K$ be a convex body in $\R^d$ which slides freely in a ball. 
Let $K^{(n)}$ denote the intersection
of $n$ closed half-spaces containing $K$ whose bounding hyperplanes
are independent and identically distributed according to a certain
prescribed probability distribution. We prove an asymptotic formula for the 
expectation of the difference of the volumes of $K^{(n)}$ and $K$, and an asymptotic
upper bound on the variance of the volume of $K^{(n)}$. 
We achieve these results by first proving similar statements
for weighted mean width approximations of convex bodies that admit a rolling ball by inscribed random polytopes 
and then by polarizing these results similarly as in \cite{BoFoHu2010}.
\end{abstract}

\maketitle

%%%%%%%%%%%%%%%%%%%%%%%
\section{Introduction and main results}
In this paper, we will investigate approximations of convex bodies by random polyhedral sets $K^{(n)}$ that arise as intersections of $n$ independent and identically distributed random closed half-spaces chosen according to a prescribed probability distribution and containing a given convex body $K$. 
In the rich theory of random polytopes, the overwhelming majority of results concern approximations of convex bodies by inscribed random polytopes. 
For a survey of this extensive theory, see for example the papers by B\'ar\'any \cite{Ba2008}, and Weil and Wieacker \cite{WeWie1993}. 
There is much less known about properties of random polytopes that contain a convex body. 

The probability model we consider has been investigated recently, for example, in B\"or\"oczky and Schneider \cite{BoSch2010} and in B\"or\"oczky, Fodor and Hug \cite{BoFoHu2010}. 
For a short overview of the history and known results on this and other similar circumscribed models, see, for example, \cite{BoFoHu2010} and the references therein. 
In particular, in \cite{BoFoHu2010} an asymptotic formula was proved for the expectation of the mean width difference of $K^{(n)}\cap K_1$ and $K$ without any smoothness assumption on the boundary of the convex body $K$, where $K_1$ denotes the radius $1$ parallel body of $K$.
Since the random polyhedral set $K^{(n)}$ is unbounded with positive probability, it is necessary to take an intersection such as, for example, $K^{(n)}\cap K_1$ to obtain a finite value for the expectation of geometric functionals like the intrinsic volumes. 
In this probability model, the role of the radius $1$ parallel body $K_1$ is not essential in the sense that if we choose another convex body in its place that contains $K$ in its interior, then this only affects the normalization constants in the theorems. 

In the following, we will prove a similar asymptotic formula for the expectation of the volume difference $\E (V(K^{(n)}\cap K_1)-V(K))$ under a mild smoothness assumption. 

In the theory of random polytopes, there is comparatively less known about the variance of random variables associated with geometric properties of random polytopes than about their means. Recently, there has been significant progress in this direction in the case of inscribed random polytopes, and also for Gaussian random polytopes. For more information and references, see B\'ar\'any \cite{Ba2008}, Calka and Yukich \cite{CaYu2014}, Calka, Schreiber and Yukich \cite{CaSchYu2013}, and Hug \cite{Hu2013}. 
However, these recently developed powerful techniques have not yet been used to establish bounds on the variance of geometric functionals associated with random polyhedral sets containing a convex body. 
In this article, using some of the methods described in B\"or\"oczky, Fodor, Reitzner and V\'\i gh \cite{BoFoReVi2009} and in B\"or\"oczky, Fodor and Hug \cite{BoFoHu2010}, we will prove an asymptotic upper bound for the variance of the volume $V(K^{(n)}\cap K_1)$. 
This asymptotic upper bound then yields a strong law of large numbers for $V(K^{(n)}\cap K_1)$.   

In order to establish these results, we first derive dual results for the mean width difference of $K$ and a random polytope $K_{(n)}$ inscribed in $K$, that is, the convex hull of $n$ independent random points from $K$ chosen according to a probability distribution.
Then we apply polarity arguments.

For a precise formulation of our results, we need the following definitions (cf. p. 156 and p. 164 in \cite{Sch2014}).
We say that the convex body $K$ slides freely in a ball $B$ if for each boundary point $p$ of $B$, there is a translate $K+v$ of $K$ with the property that $p\in K+v$ and $K+v\subset B$.
Moreover, a ball $B$ rolls freely inside $K$ if for each boundary point $p$ of $K$, there is a translate $B+v$ of $B$ such that $p \in B+v$ and $B+v \subset K$.
Note that a ball rolls freely inside a convex body $K$ if and only if it slides freely inside $K$.

Since we do only require weak differentiability assumptions on the boundary of $K$ in this article, we use generalized notions of differentiability and curvature; see Sections 1.5, 2.5 and 2.6 in Schneider \cite{Sch2014}. 
In particular, $\kappa(x)$ denotes the generalized Gaussian curvature of the boundary $\partial K$ of $K$ at $x$; precise definitions follow in the next section.

Finally, we define the constant
\[ c_d := \frac{(d\kappa_d)^{\frac{2}{d+1}}\Gamma(\frac{2}{d+1})}{(d+1)^{\frac{d-1}{d+1}}\kappa_{d-1}^{\frac{2}{d+1}}}. \]

Our main results are stated in the following theorems. 
Here, we only formulate special cases, whereas we prove more general results (see Theorems~\ref{Satz2_q}, \ref{thm:var-up_q}, \ref{thm:law-of-large-num_q}, \ref{thm:vol-mu}, and \ref{thm_var-up-vol}) involving, for example, weight functions.

The first theorem establishes an asymptotic formula for the volume difference of $K^{(n)}$ and $K$. 

\begin{theorem}\label{main}
Let $K$ be a convex body in $\R^d$ which slides freely in a ball. Then
$$ \lim_{n\to\infty} n^{\frac{2}{d+1}}\,\E (V(K^{(n)}\cap K_1)-V(K))
=c_d \int_{\partial K}\kappa(x)^{-\frac{1}{d+1}}\,\HH^{d-1}(\D x).  $$
\end{theorem}
If $K$ slides freely inside a ball of radius $R$, then $\kappa(x) \geq R^{-(d-1)}$ for $\mathcal{H}^{d-1}$ almost all $x \in \partial K$ (see the proof of Theorem~\ref{thm:vol-mu} for details).
Hence, the right-hand side of the above equation is well-defined.

The second theorem establishes an asymptotic upper bound on the variance of the volume $V(K^{(n)}\cap K_1)$.

\begin{theorem}\label{thm:var-up-circ}
Let $K$ be a convex body in $\R^d$ which slides freely in a ball. 
Then 
\begin{equation*}
\Var (V(K^{(n)}\cap K_1)) \ll n^{-\frac{d+3}{d+1}},
\end{equation*}
where the implied constant depends only on $K$. 
\end{theorem}

Finally, the following law of large numbers follows from Theorem~\ref{thm:var-up-circ} by standard arguments, using the monotonicity of $V(K^{(n)}\cap K_1)-V(K)$ in $n$. 

\begin{theorem}\label{thm:lln-circ}
Let $K$ be a convex body in $\R^d$ which slides freely in a ball. Then
\[ \lim_{n\to\infty} (V(K^{(n)}\cap K_1)-V(K))\cdot n^{\frac{2}{d+1}}
=c_d \int_{\partial K}\kappa(x)^{-\frac{1}{d+1}}\,\HH^{d-1}(\D x)  \]
with probability $1$.
\end{theorem}

Theorem 3.1 in \cite{BoFoHu2010} provides an asymptotic formula for the expectation of the weighted volume difference of $K$ and $K_{(n)}$ without any smoothness assumptions on $K$.
In analogy to this, we state the following asymptotic formula for the weighted mean width difference of $K$ and $K_{(n)}$ under a mild smoothness assumption.
In the case of uniformly distributed points in $K$, this result was already proved in \cite{BoFoReVi2009}.

The width of a convex body $K$ in a given direction is the distance between two parallel support hyperplanes of $K$ that are perpendicular to this direction.  
Averaging over all directions we obtain the mean width of $K$ which we denote by $W(K)$.

\begin{theorem}\label{Satz2}
Let $K\subset\R^d$ be a convex body with $o\in\inti K$ in which a ball rolls freely.
If $\varrho$ is a  probability density function on $K$ such that $\varrho$ is positive and continuous at each boundary point of $K$, then
\[ 
\begin{split}
& \lim_{n \rightarrow \infty} n^\frac{2}{d+1} \, \E_{\varrho}(W(K) - W(K_{(n)})) \\
&\quad = \frac{2\,c_d}{(d \kappa_d)^{\frac{d+3}{d+1}}} \int_{\partial K}{\kappa(x)^{\frac{d+2}{d+1}} \,\varrho(x)^{-\frac{2}{d+1}} \, \mathcal{H}^{d-1}(\D x)}.
\end{split} 
\]
\end{theorem}

If a ball of radius $r >0$ rolls freely inside $K$, then $\kappa(x) \leq r^{-(d-1)}$, so the integral in the statement of Theorem~\ref{Satz2} is finite.

The following theorem provides an asymptotic upper bound on the variance of the mean width $W(K_{(n)})$.

\begin{theorem}\label{thm:var-up}
With the hypotheses and notation of Theorem~\ref{Satz2}, it holds that
\[ \Var_{\varrho}(W(K_{(n)})) \ll n^{-\frac{d+3}{d+1}}, \]
where the implied constant depends only on $K$ and $\varrho$.  
\end{theorem}
A lower bound of the same order can be obtained by similar arguments as in \cite{BoFoReVi2009}.
The upper bound yields a law of large numbers for the random variable
$W(K_{(n)})$ similarly as in \cite{BoFoReVi2009}.

\begin{theorem}\label{thm:law-of-large-num}
With the hypotheses and notation of Theorem~\ref{Satz2}, it holds that
\[ \begin{split}
& \lim_{n \rightarrow \infty}\left (W(K)-W(K_{(n)})\right ) \cdot n^\frac{2}{d+1} \\
&\quad = \frac{2\,c_d}{(d \kappa_d)^{\frac{d+3}{d+1}}} \int_{\partial K}{\kappa(x)^{\frac{d+2}{d+1}} \, \varrho(x)^{-\frac{2}{d+1}}\, \mathcal{H}^{d-1}(\D x)}
\end{split} \]
with probability $1$.
\end{theorem}

In Section 3, we first obtain Theorems~\ref{Satz2} - \ref{thm:law-of-large-num} as special cases of Theorems~\ref{Satz2_q}, \ref{thm:var-up_q} and \ref{thm:law-of-large-num_q}. 
Then, more general cases of Theorems~\ref{main} - \ref{thm:lln-circ} are proved in Section 4 using polarity and Theorems~~\ref{Satz2_q}, \ref{thm:var-up_q} and \ref{thm:law-of-large-num_q}, respectively.

%%%%%%%%%%%%%%%%%%%%%%%%
\section{Preliminaries}
Henceforth, $K$ denotes a convex body in the $d$-dimensional Euclidean space $\R^d$ ($d\geq 2$), that is, a compact convex set with nonempty interior. 
We use $\langle\cdot, \cdot\rangle$ for the Euclidean scalar product and $\|\cdot\|$ for the Euclidean norm in $\R^d$. For a comprehensive treatment of the theory of convex bodies, we refer to the books by Gruber \cite{Gr2007} and Schneider \cite{Sch2014}.
The $j$-dimensional Hausdorff measure is denoted by $\HH^j$, and, in particular,
$d$-dimensional volume is denoted by $V$. The unit radius closed ball centred at the origin $o$ is $B^d$ and its boundary $\partial B^d$ is $S^{d-1}$.
We use $\kappa_d=V(B^d)$ for its volume. 
The convex hull of subsets $X_1,\ldots,X_r\subset\R^d$ and points $z_1,\ldots,z_s\in\R^d$ is denoted by $[X_1,\ldots,X_r,z_1,\ldots,z_s]$.

Recall that $\partial K$ denotes the boundary of the convex body $K$.
Let $\inti K$ be the interior of $K$. 
We say that $\partial K$ is twice differentiable in the generalized sense at $x \in\partial K$ if there exists a quadratic form $Q$ in $\R^{d-1}$ with the following property: 
If $K$ is positioned in such a way that $x=o$ and $\R^{d-1}$ is a support hyperplane of $K$ at $x$, then, in a neighbourhood of $o$, $\partial K$ is the graph of a convex function $f$ defined on a $(d-1)$-dimensional ball around $o$ in $\R^{d-1}$ satisfying
$$f(z)=\frac{1}{2}Q(z)+o(\|z\|^2),\quad \text{ as } z\to 0.$$
Here $o(\cdot)$ denotes the Landau symbol.
We call $Q$ the generalized second fundamental form of $\partial K$ at $x$, and $\kappa(x)=\det Q$ is the generalized Gaussian curvature at $x\in\partial K$. 
We refer to a point $x\in\partial K$, where $\partial K$ is twice differentiable in the generalized sense, as a normal boundary point. 
(Note that this terminology is different to that in \cite{Sch2014}.)
According to a classical result of Alexandrov (see Theorem~5.4 in \cite{Gr2007} or Theorem~2.6.1 in \cite{Sch2014}), $\partial K$ is twice differentiable in the generalized sense almost everywhere with respect to $\mathcal{H}^{d-1} \llcorner \partial K$, the $(d-1)$-dimensional Hausdorff measure restricted to $\partial K$. 

If $K$ has a rolling ball of radius $r(K)>0$, that is, any $x\in \partial K$ lies on the boundary of some Euclidean ball $B$ of radius $r(K)$ with $B\subset K$, then $K$ is smooth, that is, all support hyperplanes to $K$ are unique. 
More general, it is shown in \cite{Hug2000} that the existence of a rolling ball is equivalent to the fact that the exterior unit normal is a Lipschitz map on $\partial K$. 
In this situation, we write $\sigma_K :\partial K\to S^{d-1}$
for the Gauss map, that is, $\sigma_K(x)$ is the outer unit normal
vector of $\partial K$ at $x$. 

For a general convex body $K$, the support function $h_K:\R^d\to\R$ of $K$ is defined
as
$$h_K(u):=\max \{\langle u,x\rangle : x\in K\}, \quad u\in\R^d.$$
We also define the set
\[ D_K := \{(t,u) \in [0,\infty) \times S^{d-1} : t=h_K(u)\}. \]
The width of the convex body $K$ in the direction $u\in S^{d-1}$ is defined as
$$w_K(u):=h_K(u)+h_K(-u),$$
and the mean width of $K$ is defined as
$$W(K):=\frac{1}{d\kappa_d}\int_{S^{d-1}}w_K(u)\,\HH^{d-1}(\D u)=
\frac{2}{d\kappa_d}\int_{S^{d-1}}h_K(u)\,\HH^{d-1}(\D u).$$

Let $f:S^{d-1}\to \R$ be a measurable function. Then by the following lemma, it holds that 
$$\int_{S^{d-1}}f(u)\, \HH^{d-1}(\D u)=\int_{\partial K} f(\sigma_K(x))\kappa(x)\, \HH^{d-1}(\D x).$$
This formula was proved for convex bodies of class $\mathcal{C}_+^2$ in \cite{Sch2014} (see formula (2.62)) and used in \cite{BoFoReVi2009}.

\begin{lemma}
\label{lemma2_62}
Let $K$ be a convex body in $\R^d$ in which a ball rolls freely, and let $f$ be a measurable function on $S^{d-1}$.
Then 
\[ \int_{S^{d-1}}f(u)\, \mathcal{H}^{d-1}(\D u)= \int_{\partial K}f(\sigma_K(x))\kappa(x)\, \mathcal{H}^{d-1}(\D x).\]
\end{lemma}

\begin{proof}
Since a ball rolls freely in $K$, the map $\sigma_K$ is defined everywhere on $\partial K$ and Lipschitz continuous (see Lemma~3.3 in \cite{HugSchn2013}).
Moreover, Lemma 2.3 in \cite{Hug1996} yields that the (approximate) Jacobian of $\sigma_K$ is 
\[ \text{apJ}_{d-1}\sigma_K(x)=\kappa(x) \]
for $\mathcal{H}^{d-1}$ almost all $x\in\partial K$. 
Using Federer's coarea formula (see Theorem~3.2.12 in \cite{Federer1969}), we obtain
\begin{align*}
\int_{\partial K}f(\sigma_K(x))\kappa(x)\, \mathcal{H}^{d-1}(\D x)&=\int_{\partial K} f(\sigma_K(x))\, \text{apJ}_{d-1}\sigma_K(x)\, \mathcal{H}^{d-1}(\D x)\\
&= \int_{S^{d-1}}\int_{\sigma_K^{-1}(\{u\})}f(\sigma_K(x))\, \mathcal{H}^0(\D x)\, \mathcal{H}^{d-1}(\D u)\\
&= \int_{S^{d-1}}f(u)\, \mathcal{H}^{d-1}(\D u),
\end{align*}
where we used that for $\mathcal{H}^{d-1}$ almost all $u\in S^{d-1}$ there is exactly one $x\in\partial K$ with $u=\sigma_K(x)$ (see Theorem~2.2.11 in \cite{Sch2014}).
\end{proof}

We will use the following slightly extended statement from \cite{BoFoReVi2009} several times throughout the paper. 
\begin{lemma}\label{gamma}
Let $\beta\geq 0$ and $\omega>0$.
Let $\mu: (0,\infty) \rightarrow \R$ with $\lim_{t \rightarrow 0+} \mu(t) = 1$. 
If $g(n) \rightarrow 0$ as $n \rightarrow \infty$ and
$g(n)\geq \left ( \frac{2(\alpha+1)}{\omega} \frac{\ln n}{n}\right )^{\frac2{d+1}}$ for sufficiently large $n$ with $\alpha=\frac{2(\beta+1)}{d+1}$, then
\begin{equation*}
\int_0^{g(n)}t^{\beta}(1-\mu(t)\,\omega t^{\frac{d+1}2})^n\,\D t\sim
\frac2{(d+1)\omega^{\frac{2(\beta+1)}{d+1}}}\cdot
\Gamma\left(\frac{2(\beta+1)}{d+1}\right)n^{-\frac{2(\beta+1)}{d+1}}.
\end{equation*}
\end{lemma}
We shall apply Lemma~\ref{gamma} with $g(n) = \gamma \left(\frac{\ln n}{n}\right)^{\frac{1}{d}}$ and a constant $\gamma >0$.

The notation $\Gamma(\cdot)$ stands for Euler's gamma function. 
For real functions $f$ and $g$ defined on the same space $I \subset \R$, we write
$f\ll g$ or $f=O(g)$ if there exists a positive constant $c$, depending on $K$ and possibly other functions (such as $\varrho$ and $q$), such that 
$|f|\leq c\cdot g$ on $I$. 
We write $f\sim g$ if $I = \mathbb{N}$ and $f(n)/g(n) \rightarrow 1$ as $n \rightarrow \infty$, $n \in I$.

%%%%%%%%%%%%%%%%%%%%%%%%%%%%
\section{Weighted mean width approximation by inscribed polytopes}
Let us recall a general probability model (see \cite{BoFoHu2010}) for a random polytope inscribed in a $d$-dimensional convex body $K\subset\R^d$. 
We use the word {\em inscribed} in the sense that the resulting random polytope is contained in $K$, however, its vertices do not necessarily lie on $\partial K$. 

Let $\varrho$ be a bounded nonnegative measurable function on $K$. 
Without loss of generality, we may assume that $\int_K \varrho (x)\, \HH^d (\D x)=1$. 
We choose the random points from $K$ according to the probability measure $\PP_{\varrho, K}$ which has density $\varrho$ with respect to $\mathcal{H}^d \llcorner K$.
We denote the mathematical expectation with respect to $\PP_{\varrho, K}$ by 
$\E_{\varrho, K}$ or, if $K$ is clear from the context, then we simply use $\PP_\varrho$ and $\E_{\varrho}$. 
We also use the simplified notation $\PP_\varrho$ instead of $\PP_\varrho^{\otimes n}$.

Let $X_n:=\{x_1, \ldots, x_n\}$ be a sample of $n$ independent random points from $K$ chosen according to the probability distribution $\PP_{\varrho, K}$. 
The convex hull 
$$K_{(n)}:=[X_n]=[x_1, \ldots, x_n]$$ 
is a random polytope inscribed in $K$.

Let $q$ be a nonnegative measurable function on $\R \times S^{d-1}$.
We define the {\em weighted mean width} of a convex body $K$ as 
$$W_q(K) := \frac{2}{d \kappa_d} \int_{S^{d-1}} \int_0^{h_K(u)} q(s,u) \, \D s \, \mathcal{H}^{d-1}(\D u)$$
and call $q$ locally integrable if the integral
\[ \int_{S^{d-1}} \int_C q(s,u) \, \D s \, \mathcal{H}^{d-1}(\D u) \] 
is finite for all compact subsets $C$ of $\R$.

\subsection{Proof of Theorem~\ref{Satz2}}
In this subsection we prove the following theorem which implies Theorem~\ref{Satz2} if $q \equiv 1$.
\begin{theorem}\label{Satz2_q}
Let $K\subset\R^d$ be a convex body with $o\in\inti K$ in which a ball rolls freely.
Let $\varrho$ be a probability density function on $K$ and $q: \mathbb{R} \times S^{d-1} \rightarrow [0, \infty)$ a locally integrable function.
If $\varrho$ is positive and continuous at each boundary point of $K$ and $q$ is continuous at each point of $D_K$, then
\[ 
\begin{split}
& \lim_{n \rightarrow \infty}{ n^\frac{2}{d+1} \, \E_{\varrho}{\left (\frac{2}{d \kappa_d} \int_{S^{d-1}}{ \int_{h_{K_{(n)}}(u)}^{h_K(u)} {q(s,u) \, \D s} \, \mathcal{H}^{d-1}(\D u)}\right )}} \\
&\quad = \frac{2\,c_d}{(d \kappa_d)^{\frac{d+3}{d+1}}} \int_{\partial K}{\kappa(x)^{\frac{d+2}{d+1}} \, q(h_K(\sigma_K(x)), \sigma_K(x)) \,\varrho(x)^{-\frac{2}{d+1}} \, \mathcal{H}^{d-1}(\D x)}.
\end{split} 
\]
\end{theorem}

The quantity 
$$\E_{\varrho}{\left(\frac{2}{d \kappa_d} \int_{S^{d-1}} \int_{h_{K_{(n)}}(u)}^{h_K(u)} q(s,u) \, \D s \, \mathcal{H}^{d-1}(\D u)\right )}$$
in Theorem~\ref{Satz2_q} can be interpreted as the expectation of the weighted mean width difference of $K$ and the inscribed random polytope $K_{(n)}$, that is,
$$\E_{\varrho}\left (W_q(K)-W_q(K_{(n)})\right ).$$ 

\begin{proof}
For $u\in S^{d-1}$ and $t\in \R$, we define the 
hyperplane  $H(u,t):=\{y\in\R^d: \langle u,y\rangle = t\}$
and the closed halfspaces
$H^+(u,t):=\{y\in\R^d: \langle u,y\rangle \geq t\}$ and
$H^-(u,t):=\{y\in\R^d: \langle u,y\rangle \leq t\}$.
We also define the set $C(u,t):= K \cap H^+(u,t)$.
Let $x\in\partial K$ and $t\in (0, h_K(\sigma_K(x))$. 
Then $C(\sigma_K(x),h_K(\sigma_K(x))-t)$ is called a cap of height $t$ at $x\in\partial K$. 
 
In general, $\gamma_1,\gamma_2,\ldots$ will denote positive constants depending only on $K$, $\varrho$ and $q$. 
We will use $r(K)$ to denote the radius of a ball which rolls freely in $K$.
Without loss of generality, we may assume that $r(K)<1$.  

Let $L$ be an $i$-dimensional linear subspace in $\R^d$ in which an orthonormal
basis is fixed. The corresponding $(i-1)$-dimensional coordinate
hyperplanes (in $L$) divide $L$ into $2^i$ convex cones which
we call coordinate corners.

We start the proof of Theorem~\ref{Satz2_q} by ``conditioning'' on the event 
that the origin is contained in $K_{(n)}$. 
Then
\begin{align*} 
&\E_{\varrho}\left(\frac{2}{d \kappa_d} \int_{S^{d-1}} \int_{h_{K_{(n)}}(u)}^{h_K(u)} q(s,u) \, \D s \, \mathcal{H}^{d-1}(\D u)\right) \\
&\ = \frac{2}{d \kappa_d} \int_{K^n} \int_{S^{d-1}} \int_{h_{K_{(n)}}(u)}^{h_K(u)} q(s,u) \, \D s \, \mathcal{H}^{d-1}(\D u) \,\mathbb{P}_\varrho^{\otimes n}(\D(x_1, \dots, x_n)) \\
\begin{split}
&\ = \frac{2}{d \kappa_d} \int_{K^n} \int_{S^{d-1}} \int_{h_{K_{(n)}}(u)}^{h_K(u)} q(s,u) \, \mathbf{1}\{o \in K_{(n)}\} \, \D s \, \mathcal{H}^{d-1}(\D u) \, \mathbb{P}_\varrho^{\otimes n}(\D(x_1, \dots, x_n)) \\
&\ \quad + \frac{2}{d \kappa_d} \int_{K^n} \int_{S^{d-1}} \int_{h_{K_{(n)}}(u)}^{h_K(u)} q(s,u)\, \mathbf{1}\{o \notin K_{(n)}\} \, \D s \, \mathcal{H}^{d-1}(\D u) \, \mathbb{P}_\varrho^{\otimes n}(\D(x_1, \dots, x_n)).
\end{split}
\end{align*}
In the next step, we demonstrate that the second summand in the above formula is 
negligible. To show this, we need the following lemma. 

\begin{lemma}
\label{lemNotinKn}
There exists a constant $\gamma_1>0$, depending only on $K$ and $\varrho$, such that 
$$\mathbb{P}_{\varrho}\left(o \notin K_{(n)}\right) \leq 2^d(1 - \gamma_1)^n.$$
\end{lemma}

\begin{proof}
Let an orthonormal basis be fixed in $\R^d$ and let $\Theta_i$, $i=1,\ldots, 2^d$, be the corresponding coordinate corners. 
If $o\not\in K_{(n)}$, then the random points $X_n$ are strictly separated from $o$ by a hyperplane $H$. 
Hence there exists a coordinate corner, say $\Theta_j$, which is strictly separated from $X_n$ by $H$. 
Let the minimum probability
content of the coordinate corners be denoted by $\gamma_1>0$. Then
clearly 
$$\mathbb{P}_{\varrho}\left(o \notin K_{(n)}\right) \leq 2^d(1 - \gamma_1)^n.$$
\end{proof}
Since $h_{K_{(n)}}(u) \geq - h_K(-u)$ and $q$ is locally integrable, we obtain
\begin{align*}
& \frac{2}{d \kappa_d} \int_{K^n} \int_{S^{d-1}} \int_{h_{K_{(n)}}(u)}^{h_K(u)} q(s,u) \, \mathbf{1}\{o \notin K_{(n)}\} \, \D s \, \HH^{d-1}(\D u) \,\mathbb{P}_\varrho^{\otimes n}(\D(x_1, \dots, x_n)) \\
& \ \leq \frac{2}{d \kappa_d}\, \mathbb{P}_{\varrho}\left(o \notin K_{(n)}\right) \int_{S^{d-1}} \int_{-h_K(-u)}^{h_K(u)} q(s,u) \,  \D s \, \mathcal{H}^{d-1}(\D u) \\
& \ = O(e^{-\gamma_1 n}).
\end{align*}
Thus, in what follows, we will neglect the term that corresponds to the 
event that $o\not\in K_{(n)}$. 

Since $\varrho$ is positive and continuous at each boundary point of $K$, compactness arguments show that $\varrho$ is bounded from above and from below by positive constants in a suitable neighbourhood of $\partial K$.
Hence, choose $\varepsilon_0 >0$ such that $\varrho$ is positive on the $\varepsilon_0$-neighbourhood $U$ of $\partial K$.
Now define the positive constant $c_0 := \inf_{x \in U} \varrho(x)$.
Let $\gamma_2 := (\frac{3d}{c_0 \kappa_{d-1}})^{\frac{1}{d}}$ and let $n_0\in\N$ be so large that for all $n>n_0$ the following conditions are satisfied:
\begin{equation} 
\label{n_0}
\begin{array}{ll}   
\text{a)} & r(K) \geq \gamma_2\left(\frac{\ln n}{n}\right)^{\frac{1}{d}}, \\
\text{b)} &  \varrho \geq c_0 \text{ in } C\left(u, h_K(u) -  
\gamma_2\left(\frac{\ln n}{n}\right)^{\frac{1}{d}} \right) \text{ for all } u\in S^{d-1}, \\
\text{c)} & \left(\frac{3(d+2)}{2c_0 \kappa_{d-1} r(K)^{\frac{d-1}{2}}} \frac{\ln n}{n}\right)^{\frac{2}{d+1}} \leq  \gamma_2\left(\frac{\ln n}{n}\right)^{\frac{1}{d}}.
\end{array}  
\end{equation}

From now on we assume that $n > n_0$. 
Using Fubini's theorem and Lemma \ref{lemNotinKn}, we obtain that
\begin{align*}
&\frac{2}{d \kappa_d} \int_{K^n} \int_{S^{d-1}} \int_{h_{K_{(n)}}(u)}^{h_K(u)} q(s,u) \, \mathbf{1}\{o \in K_{(n)}\} \, \D s \, \mathcal{H}^{d-1}(\D u) \, \mathbb{P}_\varrho^{\otimes n}(\D(x_1, \dots, x_n)) \\
&\ = \frac{2}{d \kappa_d} \int_{S^{d-1}} \int_0^{h_K(u)} q(s,u) \\
&\ \quad \times \int_{K^n} \mathbf{1}\{h_{K_{(n)}}(u) < s\}\, \mathbf{1}\{o \in K_{(n)}\} \, \mathbb{P}_\varrho^{\otimes n}(\D(x_1, \dots, x_n)) \, \D s \, \mathcal{H}^{d-1}(\D u) \displaybreak[0] \\
&\ = \frac{2}{d \kappa_d} \int_{S^{d-1}} \int_0^{h_K(u)} q(s,u) \\
&\ \quad \times \int_{K^n} \mathbf{1}\{X_n \subset K\backslash C\left(u,s\right)\} \, (1-\mathbf{1}\{o \notin K_{(n)}\}) \, \mathbb{P}_\varrho^{\otimes n}(\D(x_1, \dots, x_n)) \, \D s\, \mathcal{H}^{d-1}(\D u) \\
&\ = \frac{2}{d \kappa_d} \int_{S^{d-1}} \int_0^{h_K(u)} q(s,u) \left(1 - \mathbb{P}_\varrho\left(x_1 \in C\left(u,s\right)\right)\right)^n  \D s \, \mathcal{H}^{d-1}(\D u) + O(e^{-\gamma_1 n}).
\end{align*}
Decomposing the inner integral, we get
\begin{align*}
&\frac{2}{d \kappa_d} \int_{S^{d-1}} \int_0^{h_K(u)} q(s,u)\left(1 - \mathbb{P}_\varrho\left(x_1 \in C\left(u,s\right)\right)\right)^n \D s \,  \mathcal{H}^{d-1}(\D u) \\
&\ = \frac{2}{d \kappa_d} \int_{S^{d-1}} \int_0^{h_K(u)-\gamma_2\left(\frac{\ln n}{n}\right)^{\frac{1}{d}}} q(s,u)\left(1 - \mathbb{P}_\varrho\left(x_1 \in C\left(u,s\right)\right)\right)^n \D s \, \mathcal{H}^{d-1}(\D u) \\ 
&\ \quad + \frac{2}{d \kappa_d} \int_{S^{d-1}} \int_{h_K(u)-\gamma_2\left(\frac{\ln n}{n}\right)^{\frac{1}{d}}}^{h_K(u)} q(s,u)\left(1 - \mathbb{P}_\varrho\left(x_1 \in C\left(u,s\right)\right)\right)^n \D s \, \mathcal{H}^{d-1}(\D u).
\end{align*}
We will demonstrate that the first summand is negligible. In order to show this, 
we first estimate the integrand. From (\ref{n_0}a), (\ref{n_0}b), and the choice 
of $\gamma_2$, it follows that 
\begin{align}
\label{eq:x_in_cap}
&\mathbb{P}_\varrho\left(x_1 \in C\left(u,h_K(u) -\gamma_2\left(\frac{\ln n}{n}\right)^{\frac{1}{d}}\right)\right) \\
&\ = \int_{C\left(u,h_K(u) -\gamma_2\left(\frac{\ln n}{n}\right)^{\frac{1}{d}}\right)} \varrho\left(x\right)\, \mathcal{H}^d(\D x) \notag \\
&\ \geq c_0 V\left(C\left(u,h_K(u) -\gamma_2\left(\frac{\ln n}{n}\right)^{\frac{1}{d}}\right)\right) \notag\\
&\ \geq c_0  \frac{1}{d} \kappa_{d-1} \gamma_2^d \, \frac{\ln n}{n} \notag\\
&\ = \frac{3\ln n}{n}. \notag
\end{align}
This inequality also holds for larger caps $C(u,s)$ with $s \in [0, h_K(u)- \gamma_2\left(\frac{\ln n}{n}\right)^{\frac{1}{d}}]$.
Hence, using the fact that 
$\left(1-\frac{3\ln n}{n}\right)^n\leq e^{-3\ln n}=n^{-3}$, we obtain
\begin{align*} 
&\frac{2}{d \kappa_d} \int_{S^{d-1}} \int_0^{h_K(u)-\gamma_2\left(\frac{\ln n}{n}\right)^{\frac{1}{d}}} q(s,u)\left(1 - \mathbb{P}_\varrho\left(x_1 \in C\left(u,s\right)\right)\right)^n \D s \, \mathcal{H}^{d-1}(\D u) \\
&\ \leq \frac{2}{d \kappa_d} \int_{S^{d-1}} \int_0^{h_K(u)-\gamma_2\left(\frac{\ln n}{n}\right)^{\frac{1}{d}}} q(s,u)\left(1-\frac{3 \ln n}{n}\right)^n \D s \, \mathcal{H}^{d-1}(\D u) \\
&\ \leq \frac{2}{d \kappa_d} \int_{S^{d-1}} \int_0^{h_K(u)-\gamma_2\left(\frac{\ln n}{n}\right)^{\frac{1}{d}}} q(s,u) \, n^{-3}\, \D s \, \mathcal{H}^{d-1}(\D u)\\
&\ \ll n^{-3}.
\end{align*}
Let $\widetilde{C}\left(u,t\right) := C\left(u, h_K(u)-t \right)$ and substitute $s = h_K(u) - t$.
Then
\begin{align*}
&\frac{2}{d \kappa_d} \int_{S^{d-1}} \int_{h_K(u)-\gamma_2\left(\frac{\ln n}{n}\right)^{\frac{1}{d}}}^{h_K(u)} q(s,u)\left(1 - \mathbb{P}_\varrho\left(x_1 \in C\left(u,s\right)\right)\right)^n \D s \, \mathcal{H}^{d-1}(\D u) \\
&\ = \frac{2}{d \kappa_d} \int_{S^{d-1}} \int_0^{\gamma_2\left(\frac{\ln n}{n}\right)^{\frac{1}{d}}} q(h_K(u)-t,u) \left(1 - \mathbb{P}_\varrho\left(x_1 \in \widetilde{C}\left(u,t\right)\right)\right)^n \D t \, \mathcal{H}^{d-1}(\D u).
\end{align*}
Decomposing the above integral, we get
\begin{align*}
\begin{split}
&\frac{2}{d \kappa_d} \int_{S^{d-1}} \int_0^{\gamma_2\left(\frac{\ln n}{n}\right)^{\frac{1}{d}}} q(h_K(u),u) \left(1 - \mathbb{P}_\varrho\left(x_1 \in \widetilde{C}\left(u,t\right)\right)\right)^n \D t \, \mathcal{H}^{d-1}(\D u) \\
&\ + \frac{2}{d \kappa_d} \int_{S^{d-1}} \int_0^{\gamma_2\left(\frac{\ln n}{n}\right)^{\frac{1}{d}}} \left\{ q(h_K(u)-t,u) - q(h_K(u),u) \right\}  \\
&\ \quad \times \left(1 - \mathbb{P}_\varrho\left(x_1 \in \widetilde{C}\left(u,t\right)\right)\right)^n \D t \, \mathcal{H}^{d-1}(\D u).
\end{split}
\end{align*}
We are going to show that the second summand is again negligible. 
Let $\varepsilon>0$ be fixed.
Since $q$ is continuous at each point of $D_K$, a compactness argument shows that  if $n$ is sufficiently large then $| q(h_K(u)-t,u) - q(h_K(u),u) | \leq \varepsilon$ for all $t \in \left[0, \gamma_2\left(\frac{\ln n}{n}\right)^{\frac{1}{d}}\right]$ and for all $u\in S^{d-1}$.
Hence, if $n$ is sufficiently large, then
\begin{align*}
\begin{split}
&\frac{2}{d \kappa_d} \int_{S^{d-1}} \int_0^{\gamma_2\left(\frac{\ln n}{n}\right)^{\frac{1}{d}}} \left| q(h_K(u)-t,u) - q(h_K(u),u) \right| \\
& \ \quad \times \left(1 - \mathbb{P}_\varrho\left(x_1 \in \widetilde{C}\left(u,t\right)\right)\right)^n \D t \, \mathcal{H}^{d-1}(\D u)  
\end{split} \\
&\ \leq \varepsilon\, \frac{2}{d \kappa_d} \int_{S^{d-1}} \int_0^{\gamma_2\left(\frac{\ln n}{n}\right)^{\frac{1}{d}}}  \left(1 - \mathbb{P}_\varrho\left(x_1 \in \widetilde{C}\left(u,t\right)\right)\right)^n \D t \, \mathcal{H}^{d-1}(\D u).
\end{align*}
It follows from (\ref{n_0}a) and (\ref{n_0}b) that (if $n$ is sufficiently large)
\begin{align}\label{eq:pi-rho-lower}
&\mathbb{P}_\varrho\left(x_1 \in \widetilde{C}\left(u,t\right)\right)
 = \int_{ C\left(u,h_K(u)-t\right)} \varrho(x)\, \mathcal{H}^d(\D x) \\
&\ \geq c_0 V\left( C\left(u,h_K(u)-t\right)\right)  
\geq  \frac{2c_0 \kappa_{d-1} r(K)^{\frac{d-1}{2}}  t^{\frac{d+1}{2}}}{d+1}.\notag
\end{align} 

For any fixed $u\in S^{d-1}$, Lemma~\ref{gamma} with $\beta=0$ and (\ref{n_0}c) imply 
\begin{align}
&\int_0^{\gamma_2\left(\frac{\ln n}{n}\right)^{\frac{1}{d}}} \left(1 - \mathbb{P}_\varrho\left(x_1 \in \widetilde{C}\left(u,t\right)\right)\right)^n \D t \notag \\
&\ \leq \int_0^{\gamma_2\left(\frac{\ln n}{n}\right)^{\frac{1}{d}}} \left(1- \frac{2c_0 \kappa_{d-1} r(K)^{\frac{d-1}{2}}}{d+1} \, t^{\frac{d+1}{2}}\right)^n \D t \notag \\
&\ \ll \frac{2^{\frac{d-1}{d+1}}}{(d+1)^{\frac{d-1}{d+1}}c_0^{\frac{2}{d+1}} \kappa_{d-1}^{\frac{2}{d+1}} r(K)^{\frac{d-1}{d+1}}} \Gamma\left(\frac{2}{d+1}\right) n^{-\frac{2}{d+1}}.
\label{abschIntegral}
\end{align}
Therefore
\begin{align*}
\begin{split}
&\frac{2}{d \kappa_d} \int_{S^{d-1}} \int_0^{\gamma_2\left(\frac{\ln n}{n}\right)^{\frac{1}{d}}} \left| q(h_K(u)-t,u) - q(h_K(u),u) \right| \\
& \ \quad \times \left(1 - \mathbb{P}_\varrho\left(x_1 \in \widetilde{C}\left(u,t\right)\right)\right)^n \D t \, \mathcal{H}^{d-1}(\D u) 
\end{split} \\
& \ \ll \varepsilon\, \frac{2}{d \kappa_d} \int_{S^{d-1}} \frac{2^{\frac{d-1}{d+1}}}{(d+1)^{\frac{d-1}{d+1}} c_0^{\frac{2}{d+1}} \kappa_{d-1}^{\frac{2}{d+1}} r(K)^{\frac{d-1}{d+1}}} \Gamma\left(\frac{2}{d+1}\right) n^{-\frac{2}{d+1}} \, \mathcal{H}^{d-1}(\D u) \\
& \ \ll \varepsilon \, n^{-\frac{2}{d+1}}.
\end{align*}

In summary, we have obtained that
\begin{align*}
&\E_{\varrho}\left({\frac{2}{d \kappa_d} \int_{S^{d-1}}{ \int_{h_{K_{(n)}}(u)}^{h_K(u)} q(s,u) \, \D s \, \mathcal{H}^{d-1}(\D u)}} \right) \\
&\ = \frac{2}{d \kappa_d} \int_{S^{d-1}} \int_0^{\gamma_2\left(\frac{\ln n}{n}\right)^{\frac{1}{d}}} q(h_K(u),u) \left(1 - \mathbb{P}_\varrho\left(x_1 \in \widetilde{C}\left(u,t\right)\right)\right)^n \D t \, \mathcal{H}^{d-1}(\D u) \\
& \ \quad + O\left(\varepsilon \, n^{-\frac{2}{d+1}}\right) + O\left(n^{-3}\right) + O(e^{-\gamma_1 n}).
\end{align*}

For $u \in S^{d-1}$, let 
$$ \Theta_n(u) := n^{\frac{2}{d+1}} \, \frac{2}{d \kappa_d} \, q(h_K(u),u) 
\int_0^{\gamma_2\left(\frac{\ln{n}}{n}\right)^{\frac{1}{d}}} \left(1 - \mathbb{P}_\varrho \left(x_1 \in \widetilde{C}\left(u,t\right)\right)\right)^n \D t. $$
It follows from \eqref{abschIntegral} that
\begin{align*}
\Theta_n(u) \ll 
\frac{2}{d \kappa_d} \, q(h_K(u),u) 
\frac{2^{\frac{d-1}{d+1}}}{(d+1)^{\frac{d-1}{d+1}} c_0^{\frac{2}{d+1}}\kappa_{d-1}^{\frac{2}{d+1}} 
r(K)^{\frac{d-1}{d+1}}} \Gamma\left(\frac{2}{d+1}\right).
\end{align*}
Therefore $\Theta_n(u)<\gamma_3$ for all $u\in S^{d-1}$ for some suitable constant $\gamma_3>0$. 
Furthermore, the Gaussian curvature $\kappa(x)$ is also bounded from above by $r(K)^{-(d-1)}$ for $\mathcal{H}^{d-1}$ almost all $x\in\partial K$. 
Thus, Lemma \ref{lemma2_62} and Lebesgue's dominated convergence theorem yield
\begin{align*}
& \lim_{n \rightarrow \infty} n^\frac{2}{d+1} \, \E_{\varrho} \left(\frac{2}{d \kappa_d} \int_{S^{d-1}} \int_{h_{K_{(n)}}(u)}^{h_K(u)} q(s,u) \, \D s \, \mathcal{H}^{d-1}(\D u)\right) \\
&\quad = \lim_{n \rightarrow \infty} \int_{S^{d-1}} \Theta_n(u)\, \mathcal{H}^{d-1}(\D u) \\
&\quad = \lim_{n \rightarrow \infty} \int_{\partial K} \Theta_n(\sigma_K(x)) \, \kappa(x) \, \mathcal{H}^{d-1}(\D x) \\
&\quad = \int_{\partial K} \lim_{n \rightarrow \infty} \Theta_n(\sigma_K(x)) \, \kappa(x) \, \mathcal{H}^{d-1}(\D x).
\end{align*}
It remains to calculate the limit
\begin{align*}
&  \lim_{n \rightarrow \infty} \Theta_n(\sigma_K(x))
=  \frac{2}{d \kappa_d} \, q(h_K(\sigma_K(x)),\sigma_K(x)) \\  
&\quad\quad\times\lim_{n \rightarrow \infty} n^{\frac{2}{d+1}}
\int_0^{\gamma_2\left(\frac{\ln{n}}{n}\right)^{\frac{1}{d}}} \left(1 - \mathbb{P}_\varrho\left(x_1 \in C\left(\sigma_K(x),h_K(\sigma_K(x))-t\right)\right)\right)^n \D t
\end{align*}
for $\mathcal{H}^{d-1}$ almost all $x \in \partial K$.  

We start with those normal boundary points where the Gaussian curvature is zero.
\begin{lemma}
Let $x\in\partial K$ be a normal boundary point of $K$ with $\kappa(x)=0$
and $u=\sigma_K(x)$. Then 
$$\lim_{n \rightarrow \infty}{\Theta_n(u)}=0.$$ 
\end{lemma}

\begin{proof}
It is sufficient to prove that for any given $\varepsilon>0$,
\begin{equation}\label{kappa=zero}
n^{\frac{2}{d+1}} \int_0^{\gamma_2\left(\frac{\ln{n}}{n}\right)^{\frac{1}{d}}} \left(1 - \mathbb{P}_\varrho \left(x_1 \in \widetilde{C}\left(u,t\right)\right)\right)^n \D t
\ll\varepsilon
\end{equation}
if $n$ is large enough. 
Since $\varrho$ is positive in a neighbourhood of $x$ (see (\ref{n_0}b))
$$\mathbb{P}_\varrho
\left(x_1 \in \widetilde{C}\left(u,t\right)\right)=
\int_{\widetilde{C}\left(u,t\right)}
\varrho(y) \, \HH^d(\D y)\gg V\left (\widetilde{C}\left(u,t\right )\right ),$$
for sufficiently large $n$ and $t\in 
\left ( 0,\gamma_2\left(\frac{\ln{n}}{n}\right)^{\frac{1}{d}}\right )$.

By the assumption that $\kappa(x)=0$,
one principal curvature of $\partial K$ at $x$ is zero, and hence, in particular, less than
$\varepsilon^{d+1}r(K)^{d-2}$. Thus, 
$$\HH^{d-1}(K\cap H(u, h_K(u)-t))\gg \sqrt{t\varepsilon^{-(d+1)}r(K)^{-(d-2)}}
\cdot\sqrt{tr(K)}^{d-2}.$$
Therefore
$$V(\widetilde{C}\left(u,t\right))\gg \varepsilon^{-\frac{d+1}{2}}t^{\frac{d+1}{2}}.$$
Now, Lemma~\ref{gamma} with $\beta=0$ readily implies (\ref{kappa=zero}). 
\end{proof}

Next, we are going to consider the case where $x\in\partial K$ is a normal
boundary point with $\kappa(x)>0$. 
Set $u=\sigma_K(x)$ for brevity of notation. 

Let $Q$ denote the second fundamental form of $\partial K$ as a function in the orthogonal complement $u^\perp$ of $u$ in $\R^d$. 
Let 
$$E=\{z\in u^\perp: Q(z)\leq 1\}$$
be the indicatrix of $\partial K$ at $x$.
It is well-known that if the orthonormal basis vectors $v_1, \ldots, v_{d-1}$ in $u^\perp$ are aligned with the principal directions of curvature of $\partial K$ at $x$, then $Q(z)$ can be written as
$$Q(z)=\sum_{i=1}^{d-1}k_i(x)z_i^2,$$
where the quantities $k_i=k_i(x)$, $i=1,\ldots, d-1$, are the (generalized) principal curvatures of $\partial K$ at $x$ with $z_i \in \R$ and where $z=z_1v_1+\cdots + z_{d-1}v_{d-1}$. 
There is a nondecreasing function $\mu:(0,\infty)\to\R$ with $\lim_{r\to 0^+}\mu(r)=1$ such that
\begin{equation}\label{E}
\frac{\mu(r)^{-1}}{\sqrt{2r}}\,(\widetilde{K}(u,r)+ru-x)\subset E\subset 
\frac{\mu(r)}{\sqrt{2r}}\,(\widetilde{K}(u,r)+ru-x),
\end{equation}
where $\widetilde{K}(u,r)=K\cap H(u, h_K(u)-r)$. From (\ref{E}) and Fubini's 
theorem, it follows that 
$$V(\widetilde{C}(u,r)) = V(K\cap H^+(u,h_K(u)-r))=\mu_1(r)\frac{(2r)^{\frac{d+1}{2}}}{d+1}\kappa_{d-1}\kappa (x)^{-\frac{1}{2}},
$$ 
where $\mu_1:(0,\infty)\to \R$ satisfies $\lim_{r\to 0^+}\mu_1(r)=1$.  
Now, by the continuity of $\varrho$ at $x$ and using Lemma~\ref{gamma} with
$\beta=0$, we obtain that
\begin{align*}
&\lim_{n \rightarrow \infty} \Theta_n(\sigma_K(x)) \\
&\quad = \frac{2}{d \kappa_d} \, q(h_K(\sigma_K(x)),\sigma_K(x)) \\
&\quad \quad\times   
\lim_{n \rightarrow \infty} n^{\frac{2}{d+1}}
\int_0^{\gamma_2\left(\frac{\ln{n}}{n}\right)^{\frac{1}{d}}}
\left(1 - \mathbb{P}_\varrho\left(x_1 \in C\left(\sigma_K(x),h_K(\sigma_K(x))-t\right)\right)\right)^n \D t \\
&\quad = \frac{2}{d \kappa_d} \, q(h_K(\sigma_K(x)),\sigma_K(x))\\
&\quad \quad\times \lim_{n \rightarrow \infty} n^{\frac{2}{d+1}}
\int_0^{\gamma_2\left(\frac{\ln{n}}{n}\right)^{\frac{1}{d}}} 
\left(1 - \mu_1(t)\frac{(2t)^{\frac{d+1}{2}}}{d+1}\kappa_{d-1}\kappa(x)^{-\frac{1}{2}}\varrho(x)\right)^n \D t\\
&\quad = \frac{2\Gamma\left( \frac{2}{d+1}\right )}
{d \kappa_d (d+1)^{\frac{d-1}{d+1}}\kappa_{d-1}^{\frac{2}{d+1}}} \,
q(h_K(\sigma_K(x)),\sigma_K(x))\kappa(x)^{\frac{1}{d+1}}\varrho(x)^{-\frac{2}{d+1}},
\end{align*}
and thus,
\begin{align*}
&\lim_{n\to\infty}n^{\frac{2}{d+1}}\,\E_{\varrho}\left(\frac{2}{d\kappa_d}\int_{S^{d-1}}\int_{h_{K_{(n)}(u)}}^{h_K(u)} q(s,u) \,\D s \,\HH^{d-1}(\D u)\right)\\
&\quad = \int_{\partial K}\lim_{n\to\infty} \Theta_n(\sigma_K(x))\kappa(x)\,\HH^{d-1}(\D x)\\
&\quad = \frac{2\,\Gamma\left( \frac{2}{d+1}\right )}
{d \kappa_d (d+1)^{\frac{d-1}{d+1}}\kappa_{d-1}^{\frac{2}{d+1}}} 
\int_{\partial K} q(h_K(\sigma_K(x)),\sigma_K(x)) \,
\kappa(x)^{\frac{d+2}{d+1}}\,\varrho(x)^{-\frac{2}{d+1}}\,\HH^{d-1}(\D x), 
\end{align*}
which finishes the proof of Theorem~\ref{Satz2_q}.
\end{proof}

\subsection{Upper bound on the variance: proof of Theorem~\ref{thm:var-up}}\label{var-up}
In this section, we prove the asymptotic upper bound in Theorem~\ref{thm:var-up}.
In fact, we prove the following theorem which provides an asymptotic upper bound on the variance of the weighted mean width $W_q(K_{(n)})$ and which directly implies Theorem~\ref{thm:var-up}.

\begin{theorem}\label{thm:var-up_q}
With the hypotheses and notation of Theorem~\ref{Satz2_q}, it holds that
\[ \Var_{\varrho}(W_q(K_{(n)})) \ll n^{-\frac{d+3}{d+1}}, \]
where the implied constant depends only on $K$, $q$ and $\varrho$.  
\end{theorem}
A lower bound of the same order can be obtained by the same arguments as in \cite{BoFoReVi2009}.

\begin{proof}
Our argument is similar to the one presented in B\"or\"oczky, Fodor, Reitzner and V\'\i gh \cite{BoFoReVi2009}. 
The main tool is the Efron-Stein jackknife inequality (cf. Reitzner \cite{Reitzner2003})
\begin{equation}\label{efron-stein}
 \Var_{\varrho} (W_q(K_{(n)})) \leq (n+1)\,\E_\varrho (W_q(K_{(n+1)})-W_q(K_{(n)}))^2.
\end{equation}
It follows from \eqref{efron-stein} and Fubini's theorem that
\begin{align*}
&\Var_{\varrho}(W_q(K_{(n)})) \\
\begin{split}
&\ \ll n \int_{K^{n+1}} \left( \int_{S^{d-1}} \int_{h_{K_{(n)}}(u)}^{h_{K_{(n+1)}}(u)} q(s,u) \, \D s \, \mathcal{H}^{d-1}(\D u) \right) \\
&\ \quad \times \left( \int_{S^{d-1}} \int_{h_{K_{(n)}}(v)}^{h_{K_{(n+1)}}(v)} q(t,v) \, \D t \, \mathcal{H}^{d-1}(\D v) \right) \, \mathbb{P}_\varrho^{\otimes (n+1)}(\D X_{n+1}) 
\end{split}\\
\begin{split}
&\ = n \int_{S^{d-1}} \int_{S^{d-1}} \int_0^{h_K(u)} \int_0^{h_K(v)} q(s,u)\, q(t,v) \int_{K^{n+1}} \ii\{h_{K_{(n)}}(u) \leq s \leq h_{K_{(n+1)}}(u)\} \\
&\ \quad \times \ii\{h_{K_{(n)}}(v) \leq t \leq h_{K_{(n+1)}}(v)\} \\ 
&\ \quad \times \ii\{o \in K_{(n)} \}\, \mathbb{P}_\varrho^{\otimes (n+1)}(\D X_{n+1})  \, \D t \,\D s \, \mathcal{H}^{d-1}(\D v) \, \mathcal{H}^{d-1}(\D u) + O(n e^{-\gamma_1 n})
\end{split} \\
\begin{split}
&\ = n \int_{S^{d-1}} \int_{S^{d-1}} \int_0^{h_K(u)} \int_0^{h_K(v)} q(s,u)\, q(t,v) \int_{K^{n+1}} \ii\{x_{n+1} \in C(u,s) \cap C(v,t)\} \\
&\ \quad \times \ii\{X_n \subset K \setminus( C(u,s) \cup C(v,t))\} \\
&\ \quad \times \ii\{o \in K_{(n)} \}\, \mathbb{P}_\varrho^{\otimes (n+1)}(\D X_{n+1})  \, \D t \, \D s \, \mathcal{H}^{d-1}(\D v) \, \mathcal{H}^{d-1}(\D u) + O(n e^{-\gamma_1 n})
\end{split} \\
\begin{split}
&\ = n \int_{S^{d-1}} \int_{S^{d-1}} \int_0^{h_K(u)} \int_0^{h_K(v)} q(s,u)\, q(t,v) \, \mathbb{P}_\varrho(x_{n+1} \in C(u,s) \cap C(v,t)) \\
&\ \quad \times (1- \mathbb{P}_\varrho(x_1 \in C(u,s) \cup C(v,t)))^n  \, \D t \, \D s \, \mathcal{H}^{d-1}(\D v) \, \mathcal{H}^{d-1}(\D u) + O(n e^{-\gamma_1 n}).
\end{split}\\
\end{align*}
Now, for $b \geq 0$, $0 \leq s \leq h_K(u)$ and $u \in S^{d-1}$ let
\[ \Sigma(u,s;b) = \{ v \in S^{d-1} : C(u,s) \cap \widetilde{C}(v,b) \neq \emptyset \}, \]
and for $v \in \Sigma(u,s;b)$ let
\[ \mathbb{P}_\varrho^+(u,s;v,b) = \max\{\mathbb{P}_\varrho(x_1 \in C(u,s)),\mathbb{P}_\varrho(x_1 \in \widetilde{C}(v,b))\}. \]

Let $\gamma_2$ be defined as on page 7.
By symmetry we may assume $s\leq t$.
Then substituting $b=h_K(v)-t$ and splitting the domain of integration of $s$, we obtain that
\begin{align*}
&\Var_{\varrho}(W_q(K_{(n)})) \\
\begin{split}
&\ \ll n \int_{S^{d-1}} \int_{S^{d-1}} \int_0^{h_K(u)} \int_s^{h_K(v)} q(s,u)\, q(t,v) \, \mathbb{P}_\varrho(x_{n+1} \in C(u,s) \cap C(v,t)) \\
&\ \quad \times (1- \mathbb{P}_\varrho(x_1 \in C(u,s) \cup C(v,t)))^n  \, \D t \, \D s \, \mathcal{H}^{d-1}(\D v) \, \mathcal{H}^{d-1}(\D u) + O(n e^{-\gamma_1 n})
\end{split} \\
\begin{split}
&\ \ll n \int_{S^{d-1}} \int_0^{h_K(u)} \int_0^{h_K(u)-s} \int_{\Sigma(u,s;b)} q(s,u)\, q(h_K(v)-b,v) \, \mathbb{P}_\varrho^+(u,s;v,b)  \\
&\ \quad \times (1- \mathbb{P}_\varrho^+(u,s;v,b))^n \, \mathcal{H}^{d-1}(\D v) \, \D b \, \D s \, \mathcal{H}^{d-1}(\D u) + O(n e^{-\gamma_1 n})
\end{split} \\
\begin{split}
&\ = n \int_{S^{d-1}} \int_0^{h_K(u)-\gamma_2(\frac{\ln n}{n})^{\frac{1}{d}}} \int_0^{h_K(u)-s} \int_{\Sigma(u,s;b)} q(s,u)\, q(h_K(v)-b,v) \, \mathbb{P}_\varrho^+(u,s;v,b)  \\
&\ \quad \times (1- \mathbb{P}_\varrho^+(u,s;v,b))^n \, \mathcal{H}^{d-1}(\D v) \, \D b \, \D s \, \mathcal{H}^{d-1}(\D u)
\end{split} \\
\begin{split}
&\ \ + n \int_{S^{d-1}} \int_{h_K(u)-\gamma_2(\frac{\ln n}{n})^{\frac{1}{d}}}^{h_K(u)} \int_0^{h_K(u)-s} \int_{\Sigma(u,s;b)} q(s,u)\, q(h_K(v)-b,v) \, \mathbb{P}_\varrho^+(u,s;v,b)  \\
&\ \quad \times (1- \mathbb{P}_\varrho^+(u,s;v,b))^n \, \mathcal{H}^{d-1}(\D v) \, \D b \, \D s \, \mathcal{H}^{d-1}(\D u) + O(n e^{-\gamma_1 n}).
\end{split}
\end{align*}

We continue the argument by estimating the first summand in the above integral.
To achieve this estimate we will use the following inequality.
If $\alpha \in [0,1]$, then 
\[ \frac{(1-\frac{2\alpha}{3})^n}{(1-\alpha)^n} \geq \left(1+ \frac{\alpha}{3}\right)^n \geq \frac{\alpha n}{3}, \]
which yields
\begin{equation}
\label{AbschA}
\alpha(1-\alpha)^n \leq \frac{3}{n}\left(1- \frac{2\alpha}{3}\right)^n.
\end{equation}

It follows from \eqref{AbschA} and \eqref{eq:x_in_cap} that for sufficiently large $n$
\begin{align*}
\begin{split}
& n \int_{S^{d-1}} \int_0^{h_K(u)-\gamma_2(\frac{\ln n}{n})^{\frac{1}{d}}} \int_0^{h_K(u)-s} \int_{\Sigma(u,s;b)} q(s,u)\, q(h_K(v)-b,v) \, \mathbb{P}_\varrho^+(u,s;v,b)  \\
& \ \quad \times \left (1- \mathbb{P}_\varrho^+(u,s;v,b)\right )^n \mathcal{H}^{d-1}(\D v) \, \D b \, \D s \, \mathcal{H}^{d-1}(\D u)
\end{split} \\
\begin{split}
& \ \leq n \int_{S^{d-1}} \int_0^{h_K(u)-\gamma_2(\frac{\ln n}{n})^{\frac{1}{d}}} \int_0^{h_K(u)-s} \int_{\Sigma(u,s;b)} q(s,u)\, q(h_K(v)-b,v) \\
& \ \quad \times \frac{3}{n}\left (1- \frac{2}{3}\mathbb{P}_\varrho^+(u,s;v,b)\right )^n \mathcal{H}^{d-1}(\D v) \, \D b \, \D s \, \mathcal{H}^{d-1}(\D u)
\end{split} \displaybreak[0] \\
\begin{split}
& \ \leq \int_{S^{d-1}} \int_0^{h_K(u)-\gamma_2(\frac{\ln n}{n})^{\frac{1}{d}}} \int_0^{h_K(u)-s} \int_{\Sigma(u,s;b)} q(s,u)\, q(h_K(v)-b,v) \\
& \ \quad \times 3\left (1- \frac{2}{3}\frac{3 \ln n}{n}\right )^n \mathcal{H}^{d-1}(\D v) \, \D b \, \D s \, \mathcal{H}^{d-1}(\D u)
\end{split} \\
& \ \ll n^{-2}.
\end{align*}

This implies, together with \eqref{AbschA} and the substitution $a=h_K(u)-s$, that
\begin{align*}
&\Var_{\varrho}(W_q(K_{(n)})) \\
\begin{split}
&\ \ll n \int_{S^{d-1}} \int_{h_K(u)-\gamma_2(\frac{\ln n}{n})^{\frac{1}{d}}}^{h_K(u)} \int_0^{h_K(u)-s} \int_{\Sigma(u,s;b)} q(s,u)\, q(h_K(v)-b,v) \, \mathbb{P}_\varrho^+(u,s;v,b)  \\
&\ \quad \times \left (1- \mathbb{P}_\varrho^+(u,s;v,b)\right )^n \mathcal{H}^{d-1}(\D v) \, \D b \, \D s \, \mathcal{H}^{d-1}(\D u) + O(n e^{-\gamma_1 n}) + O\left(n^{-2}\right)
\end{split}\\
\begin{split}
&\ \leq n \int_{S^{d-1}} \int_{h_K(u)-\gamma_2(\frac{\ln n}{n})^{\frac{1}{d}}}^{h_K(u)} \int_0^{h_K(u)-s} \int_{\Sigma(u,s;b)} q(s,u)\, q(h_K(v)-b,v) \\
&\ \quad \times \frac{3}{n}\left (1- \frac{2}{3}\mathbb{P}_\varrho^+(u,s;v,b)\right )^n \mathcal{H}^{d-1}(\D v) \, \D b \, \D s \, \mathcal{H}^{d-1}(\D u) + O(n e^{-\gamma_1 n}) + O\left(n^{-2}\right)
\end{split}\\
\begin{split}
&\ \ll \int_{S^{d-1}} \int_0^{\gamma_2(\frac{\ln n}{n})^{\frac{1}{d}}} \int_0^a \int_{\Sigma(u,h_K(u)-a;b)} q(h_K(u)-a,u)\, q(h_K(v)-b,v) \\
&\ \quad \times \left (1- \frac{2}{3}\,\mathbb{P}_\varrho^+(u,h_K(u)-a;v,b)\right )^n \mathcal{H}^{d-1}(\D v) \, \D b \, \D a \, \mathcal{H}^{d-1}(\D u) \\
&\ \quad + O(n e^{-\gamma_1 n}) + O\left(n^{-2}\right).
\end{split}
\end{align*}

By the continuity of $q$ at each point of $D_K$, we may assume that 
if $n$ is sufficiently large then $q(h_K(u) -a,u)$ is bounded for all $a \in [0,\gamma_2(\frac{\ln n}{n})^{\frac{1}{d}}]$ and for all $u \in S^{d-1}$.
From \cite{BoFoReVi2009} (cf. Proof of the upper bound in Theorem~1.2 on page 2291) it follows that the $(d-1)$-measure of $\Sigma(u,h_K(u)-a;b)$ is at most $\gamma_4 a^{\frac{d-1}{2}}$, where the constant $\gamma_4 > 0$ depends on $d$.
Thus, using \eqref{eq:pi-rho-lower} and Lemma~\ref{gamma} with $\beta = \frac{d+1}{2}$, we obtain that
\begin{align*}
&\Var_{\varrho}(W_q(K_{(n)})) \\
\begin{split}
& \ \ll \int_{S^{d-1}} \int_0^{\gamma_2(\frac{\ln n}{n})^{\frac{1}{d}}} \int_0^a \int_{\Sigma(u,h_K(u)-a;b)} \left (1- \frac{2}{3}\,\mathbb{P}_\varrho^+(u,h_K(u)-a;v,b)\right)^n \mathcal{H}^{d-1}(\D v)\\
&\ \quad \times \D b \, \D a \, \mathcal{H}^{d-1}(\D u) + O(n e^{-\gamma_1 n}) + O\left(n^{-2}\right)
\end{split} \\
& \ \ll \int_{S^{d-1}} \int_0^{\gamma_2(\frac{\ln n}{n})^{\frac{1}{d}}} 
\int_0^a a^{\frac{d-1}{2}} \left(1- \frac{4c_0 \kappa_{d-1} r(K)^{\frac{d-1}{2}}}{3(d+1)} \,a^{\frac{d+1}{2}}\right)^n \D b \\
&\ \quad \times \D a \, \mathcal{H}^{d-1}(\D u) + O(n e^{-\gamma_1 n}) + O\left(n^{-2}\right) \displaybreak[0] \\
& \ = \int_{S^{d-1}} \int_0^{\gamma_2(\frac{\ln n}{n})^{\frac{1}{d}}} a^{\frac{d+1}{2}} \left(1- \frac{4c_0 \kappa_{d-1} r(K)^{\frac{d-1}{2}}}{3(d+1)} \,a^{\frac{d+1}{2}}\right)^n \D a \, \mathcal{H}^{d-1}(\D u) \\
&\ \quad + O(n e^{-\gamma_1 n}) + O\left(n^{-2}\right) \displaybreak[0] \\
& \ \ll \int_{S^{d-1}} 
n^{-\frac{d+3}{d+1}} \, \mathcal{H}^{d-1}(\D u) + O(n e^{-\gamma_1 n}) + O\left(n^{-2}\right) \\
& \ \ll n^{-\frac{d+3}{d+1}}.
\end{align*}

This finishes the proof of the asymptotic upper bound in Theorem~\ref{thm:var-up_q}.
\end{proof} 

\subsection{The strong law of large numbers: proof of Theorem~\ref{thm:law-of-large-num}}
The upper bound on the variance implies a law of large numbers for $W(K_{(n)})$ as stated in Theorem~\ref{thm:law-of-large-num}. 

The same holds true for the upper bound on the variance of the weighted mean width $W_q(K_{(n)})$. 
Hence, we prove the following theorem of which Theorem~\ref{thm:law-of-large-num} is a special case.

\begin{theorem}\label{thm:law-of-large-num_q}
With the hypotheses and notation of Theorem~\ref{Satz2_q}, it holds that
\[ \begin{split}
& \lim_{n \rightarrow \infty}\left (W_q(K)-W_q(K_{(n)})\right ) \cdot n^\frac{2}{d+1} \\
&\quad = \frac{2\,c_d}{(d \kappa_d)^{\frac{d+3}{d+1}}} \int_{\partial K}{\kappa(x)^{\frac{d+2}{d+1}} \, q(h_K(\sigma_K(x)), \sigma_K(x))\, \varrho(x)^{-\frac{2}{d+1}}\, \mathcal{H}^{d-1}(\D x)}
\end{split} \]
with probability $1$.
\end{theorem} 

\begin{proof}
The proof follows essentially the same argument as that of Theorem~1.3 in \cite{BoFoReVi2009}.
Here we use the more general variance estimate provided in Theorem~\ref{thm:var-up_q}.
Then for $\varepsilon > 0$, we obtain from Chebyshev's inequality that
\begin{align*}  
& \PP_\varrho\left(\left| W_q(K)-W_q(K_{(n)}) - 
\E_{\varrho}\left ({W_q(K)-W_q(K_{(n)})}\right ) \right| \cdot n^\frac{2}{d+1} \geq \varepsilon \right) \\
&\ \leq \left(\varepsilon n^{-\frac{2}{d+1}} \right)^{-2} \Var_{\varrho}(W_q(K_{(n)}))  \\
&\ \ll \varepsilon^{-2} \, n^\frac{4}{d+1}\, n^{-\frac{d+3}{d+1}} \\
&\ \ll n^{-\frac{d-1}{d+1}}.
\end{align*}

Now, the proof may be finished in exactly the same way as, for example, on page 2295 in \cite{BoFoReVi2009} using the fact that $W_q(K)-W_q(K_{(n)})$ is monotonically decreasing with $n$.
\end{proof}

\section{Polarity and circumscribed random polytopes}
In this section, we will prove generalizations of Theorems~\ref{main} and \ref{thm:var-up-circ} with the help of polarity and Theorems~\ref{Satz2_q} and \ref{thm:var-up_q}, respectively. We will  follow a  similar reasoning as in \cite{BoFoHu2010}, however, with some modifications and supplements. For the sake of completeness, we begin with some notations and we repeat some of the statements originally proved in \cite{BoFoHu2010} that will be used in the present proof as 
well. 

The polar $K^*$ of a convex body $K$ in $\R^d$ is the closed, convex set $K^*:=\{y\in\R^d: \langle x,y\rangle\leq 1 \text{ for all }x\in K\}$. We assume that $o\in \inti K$, and so $K^*$ is also a convex body with $o\in\inti K^*$. For more information see \cite{Sch2014}. 

We fix a convex body $K\subset\R^d$ with  $o\in \inti K$ and describe the particular probability model we use for constructing the random polyhedral set $K^{(n)}$ in more detail. 
Let the radius $1$ parallel body of $K$ be denoted by $K_1=K+B^d$, and let $\HH$ be the space of hyperplanes in $\R^d$ with their usual topology. 
We denote by $\HH_K$ the subspace of $\HH$ whose elements intersect $K_1$ and are disjoint from the interior of $K$. 
For $H\in\HH_K$, let $H^-$ be the closed halfspace containing $K$. 
We assume that $\mu$ is the (unique) rigid motion invariant Borel measure on $\HH$ 
which is normalized such that
$\mu(\{H\in\HH: H\cap M\neq\emptyset\})=W(M)$
for each convex body $M$ in $\R^d$. 
Let $2\mu_K$ be the restriction of the measure $\mu$ onto $\HH_K$.
Then $\mu_K$ is a probability measure on $\HH_K$. 
Let $H_1,\ldots, H_n$ be independent random hyperplanes in $\R^d$, that is, independent $\HH$-valued random variables with distribution $\mu_K$, which are defined  on some suitable probability space. 
The intersection
$$K^{(n)}:=\bigcap_{i=1}^{n}H_i^{-}$$
with $H_i\in\HH_K$, for $i=1,\ldots, n$, is a random polyhedral set containing $K$.
Note that $K^{(n)}$ may be unbounded. 

Let $K$ be fixed as before. More general and with the same notations as in \cite{BoFoHu2010}, let $q:[0,\infty)\times S^{d-1}\to [0,\infty)$ denote a locally integrable function, and let 
\begin{equation}
\mu_q:=\frac{1}{d\kappa_d}\int_{S^{d-1}}\int_{0}^{\infty} \ii\{H(u,t)\in\cdot\}\, q(t,u)\, \D t \, \mathcal{H}^{d-1}(\D u).
\end{equation}
We assume that $q$ is
\begin{itemize}
\item[i)] concentrated on  $D_K^1:=\{(t,u)\in [0,\infty) \times S^{d-1}: h_K(u)\leq t\leq h_{K_1}(u)\}$, 

\item[ii)] positive and continuous at each point of 
the set $D_K$ in $D_K^1$, and

\item[iii)] satisfies $\mu_q(\HH_K)=1$. 
\end{itemize}
Then $\mu_q$ is a probability distribution of hyperplanes which is concentrated on $\mathcal{H}_K$. As before we write $H_1,\ldots,H_n$ for independent 
random hyperplanes following this distribution and $K^{(n)}$ for the intersection of the halfspaces containing $K$.  

\subsection{Proof of Theorem~\ref{main}}
In this subsection we prove Theorem \ref{thm:vol-mu}, which directly 
implies Theorem~\ref{main} in the case that $q \equiv 1 \equiv \lambda$ (in the notation of that theorem). 

In addition to the support function of a convex body $M\subset\R^d$, we now also need the radial function $\rho_M=\rho(M,\cdot):\R^d\setminus\{o\}\to[0,\infty)$ 
of $M$, but then we always assume that $o\in\inti M$. The radial function of $M$ is defined by $\rho(M,x):=\max\{t\ge 0:tx\in M\}$ for $x\in\R^d\setminus\{o\}$. For basic properties of radial functions and their connection 
to support functions, we refer to \cite{Sch2014}. 

Let $\lambda: \R \times S^{d-1} \rightarrow [0,\infty)$ be a measurable function. Then we define the weighted volume (i.e., the $\lambda$-weighted volume) of $M$ as
\[ V_\lambda(M) := \int_{S^{d-1}} \int_0^{\rho(M,u)} t^{d-1} \, \lambda(t,u) \, \D t \, \mathcal{H}^{d-1}(\D u). \]
We call $\lambda$ locally integrable, if the weighted volumes of all convex bodies $M$ with $o \in \inti M$ are finite. 

\begin{theorem} \label{thm:vol-mu}
Let $K\subset\R^d$ be a convex body with $o\in\inti K$ which slides freely inside a ball. 
Assume that the function
$q:[0,\infty)\times S^{d-1}\to [0,\infty)$ satisfies properties i), ii), and iii) as described above.
Let $\lambda: \R \times S^{d-1} \rightarrow [0,\infty)$ be a locally integrable function which is continuous at each point of the set $\{(\rho(K,u),u):u\in S^{d-1}\}$.
Then 
\begin{multline*}
\lim_{n\to\infty}n^{\frac{2}{d+1}}\, \E_{\mu_q} (V_\lambda(K^{(n)}\cap K_1)-V_\lambda(K))\\
=c_d\int_{\partial K} q(h_K(\sigma_K(x)), \sigma_K(x))^{-\frac{2}{d+1}}\,\lambda(\|x\|,x/\|x\|)\,\kappa(x)^{-\frac{1}{d+1}}\,\HH^{d-1}(\D x)
\end{multline*}
is finite and the constant $c_d$ is defined as in Theorem \ref{main}.
\end{theorem}

\begin{proof}
Let the nonnegative and measurable functional
$$F_\lambda(P):=\ii\{P\subset K_1\}(V_\lambda(P)-V_\lambda(K))$$
be defined for polyhedral sets $P$ in $\R^d$. 
Let $x_1, \dots, x_n \in K^*\backslash (K_1)^*$.
Then $o \in \inti K \subset [x_1, \ldots, x_n]^* \subset K_1$ and hence $o \in \inti\, [x_1, \ldots, x_n]$.
Using the substitution $s=\frac{1}{t}$ we obtain that
\begin{align*}
&F_\lambda([x_1, \ldots, x_n]^*)\\
&\ =\ii\{[x_1, \ldots, x_n]^*\subset K_1\}(V_\lambda([x_1, \ldots, x_n]^*)-V_\lambda(K))\\
&\ =\ii\{[x_1, \ldots, x_n]^*\subset K_1\}\int_{S^{d-1}}\int_{\rho(K,u)}^{\rho([x_1, \ldots, x_n]^*,u)} t^{d-1} \, \lambda(t,u) \,\D t\,\HH^{d-1}(\D u)\\
&\ =\ii\{[x_1, \ldots, x_n]^*\subset K_1\}\int_{S^{d-1}}\int_{h_{K^*}(u)^{-1}}^{h_{[x_1, \ldots, x_n]}(u)^{-1}}\,t^{d-1}\, \lambda(t,u)\,\D t\,\HH^{d-1}(\D u)\\
&\ =\ii\{[x_1, \ldots, x_n]^*\subset K_1\}\int_{S^{d-1}}\int^{h_{K^*}(u)}_{h_{[x_1, \ldots, x_n]}(u)}\widetilde{\lambda} (s,u)\,\D s\,\HH^{d-1}(\D u),
\end{align*}
where $\widetilde{\lambda} (s,u):=\ii\{s\ge h_{(K_1)^*}(u)\}\,s^{-(d+1)} \, \lambda(s^{-1},u)$ if $s>0$ and zero otherwise.

It was proved in \cite{BoSch2010} that $\PP_{\mu_q}(K^{(n)}\not\subset K_1)\ll \alpha^n$
for some real number $\alpha\in (0,1)$ depending on the convex body $K$ and the density $q$. Since the distributions of the random polyhedral sets $K^{(n)}$, based on $K$ and $q$, and $(K^*_{(n)})^*:=((K^*)_{(n)})^*$, based on $K^*$ and $\varrho$ (to be defined below), are the same (see Proposition 5.1 in \cite{BoFoHu2010} for a precise statement), it follows that 
\begin{align*}
&\E_{\mu_q}(V_\lambda(K^{(n)}\cap K_1)-V_\lambda(K))\\
&\ =\E_{\mu_q}(\ii\{K^{(n)}\subset K_1\}(V_\lambda(K^{(n)})-V_\lambda(K)))+O(\alpha^n)\\
&\ =\E_{\varrho, K^*}(\ii\{(K^*_{(n)})^*\subset K_1\}(V_\lambda((K_{(n)}^*)^*)-V_\lambda(K)))+O(\alpha^n)\\
&\ =\E_{\varrho, K^*}\left (\int_{S^{d-1}}\int^{h_{K^*}(u)}_{h_{[x_1, \ldots, x_n]}(u)}\widetilde{\lambda} (s,u)\,\D s\,\HH^{d-1}(\D u)\right )+O(\alpha^n),
\end{align*}
where $\varrho$ is defined as on page 516 in \cite{BoFoHu2010}, namely
\[ \varrho(x):= \begin{cases}
(d\kappa_d)^{-1} \, \tilde{q}(x) \, \|x\|^{-(d+1)}, & x \in K^* \backslash (K_1)^*, \\
0, & x \in (K_1)^*,
\end{cases} \]
and 
\[ \tilde{q}(x) := q\left( \frac{1}{\|x\|}, \frac{x}{\|x\|} \right), \quad x \in K^* \backslash \{o\}. \]
It is easy to check from the assumptions on $q$ that $\varrho$ is a probability density on $K^*$, which is positive and continuous at each point of $\partial K^*$. 
Moreover, the assumptions on $\lambda$ imply that $\widetilde{\lambda}$ is locally integrable and continuous 
at each point of $D_{K^*}$. 
Since Proposition~\ref{propmod} shows that $K^*$ has a rolling ball, Theorem~\ref{Satz2_q} can be applied with $K^*,q$ and $\varrho$ as defined here. This yields that
\begin{align*}
&\lim_{n\to\infty} n^{\frac{2}{d+1}}\,\E_{\mu_q}(V_\lambda(K^{(n)}\cap K_1)-V_\lambda(K))\\
&\quad=\lim_{n\to\infty}n^{\frac{2}{d+1}}\,\E_{\varrho, K^*}\left (\int_{S^{d-1}}\int^{h_{K^*}(u)}_{h_{[x_1, \ldots, x_n]}(u)} \widetilde{\lambda} (s,u)\,\D s\,\HH^{d-1}(\D u)\right )\\
&\quad=c_d(d\kappa_d)^{-\frac{2}{d+1}}\int_{\partial K^*}\widetilde{\lambda}(h_{K^*}(\sigma_{K^*}(x)),\sigma_{K^*}(x))\, \varrho(x)^{-\frac{2}{d+1}}\, \kappa^*(x)^{\frac{d+2}{d+1}}\, \HH^{d-1}(\D x)\\
&\quad=c_d\int_{\partial K^*}(h_{K^*}(\sigma_{K^*}(x)))^{-(d+1)}\lambda(h_{K^*}(\sigma_{K^*}(x))^{-1},\sigma_{K^*}(x))\,\tilde{q}(x)^{-\frac{2}{d+1}} \,\|x\|^2 \\
&\quad\quad \times \kappa^*(x)^{\frac{d+2}{d+1}}\,\HH^{d-1}(\D x),
\end{align*}
where $\kappa^*(x)$ denotes the generalized curvature of $\partial K^*$ in $x$.
Applying Lemma~6.1 from \cite{BoFoHu2010} (cf. p. 519 and the notation and terminology used in \cite{BoFoHu2010}), we obtain
\begin{align*}
&\lim_{n\to\infty} n^{\frac{2}{d+1}}\,\E_{\mu_q}(V_\lambda(K^{(n)}\cap K_1)-V_\lambda(K))\\
&\quad=c_d\int_{S^{d-1}}\rho(K, \sigma_{K^*}(\nabla h_{K^*}(u)))^{d+1}\,\tilde{q}(\nabla h_{K^*}(u))^{-\frac{2}{d+1}}\,\|\nabla h_{K^*}(u)\|^2\\
&\quad\quad\times \lambda(\rho(K,\sigma_{K^*}(\nabla h_{K^*}(u))),\sigma_{K^*}(\nabla h_{K^*}(u)))\,\kappa^*(\nabla h_{K^*}(u))\\
&\quad\quad\times D_{d-1}h_{K^*}(u)^{\frac{d}{d+1}}\,\HH^{d-1}(\D u),
\end{align*}
where $D_{d-1}h_{K^*}(u)$ denotes the product of the principal radii of curvature of $K^*$ in direction $u$ for $\mathcal{H}^{d-1}$ almost all $u \in S^{d-1}$.
Since a ball rolls freely inside $K^*$, we have $\sigma_{K^*}(\nabla h_{K^*}(u))=u$ for $\mathcal{H}^{d-1}$ almost all $u\in S^{d-1}$. Moreover, for 
$\mathcal{H}^{d-1}$ almost all $u\in S^{d-1}$, the support function  $h_{K^*}$ of $K^*$ at $u$ is second order differentiable and $\nabla h_{K^*}(u)$ 
is a normal boundary point of $K^*$. Hence, combining Lemma 3.1 and Lemma 3.4 in \cite{Hug1998}, we conclude for any such $u$ that 
$$
\kappa^*(\nabla h_{K^*}(u))=D_{d-1}h_{K^*}(u)^{-1}\in (0,\infty).
$$
(Note that if $K^*$ has a rolling ball of radius $r(K^*)>0$, then $D_{d-1}h_{K^*}(u)$ is uniformly bounded from below by $r(K^*)^{d-1}$ whenever it is defined.)  
Thus
\begin{align*}
&\lim_{n\to\infty} n^{\frac{2}{d+1}}\,\E_{\mu_q}(V_\lambda(K^{(n)}\cap K_1)-V_\lambda(K))\\
&\quad=c_d\int_{S^{d-1}}\rho(K, u)^{d+1}\,\tilde{q}(\nabla h_{K^*}(u))^{-\frac{2}{d+1}}\,\lambda(\rho(K,u),u)\,\|\nabla h_{K^*}(u)\|^2\\
&\quad\quad\times D_{d-1}h_{K^*}(u)^{-\frac{1}{d+1}}\,\HH^{d-1}(\D u).
\end{align*}

According to Theorem~2.2 in \cite{Hug1996b} for $\mathcal{H}^{d-1}$ almost all $u\in S^{d-1}$ the support function of $K^*$ is second order differentiable in $u$ and $x=\rho(K,u)u\in\partial K$ is a normal boundary point of $K$. 
Moreover, for each such $u\in S^{d-1}$, 
$$D_{d-1}h_{K^*}(u)^{\frac{1}{d+1}}=\kappa(x)^{\frac{1}{d+1}}\langle u, \sigma_{K}(x)\rangle^{-1}.$$
(Since $D_{d-1}h_{K^*}(u)$ is uniformly bounded from below by $r(K^*)^{d-1}$ it follows that
\[ \kappa(x) \geq \langle x, \sigma_{K}(x)\rangle^{d+1} \, r(K^*)^{d-1} \geq r_0^{d+1} \, r(K^*)^{d-1}  >0, \]
where we used that $\langle x, \sigma_{K}(x)\rangle = h_K(\sigma_K(x)) \geq r_0$ for a constant $r_0 >0$.
Thus, $\kappa(x)$ is uniformly bounded from below.)

With the help of the above, we obtain that
\begin{align*}
&\lim_{n\to\infty} n^{\frac{2}{d+1}}\,\E_{\mu_q}(V_\lambda(K^{(n)}\cap K_1)-V_\lambda(K))\\
&\quad=c_d\int_{S^{d-1}}\rho(K, u)^{d+1}\,\tilde{q}(\nabla h_{K^*}(u))^{-\frac{2}{d+1}}\,\lambda(\rho(K,u),u)\,\|\nabla h_{K^*}(u)\|^2\\
&\quad\quad\times \kappa(\rho(K,u)u)^{-\frac{1}{d+1}} \, \langle u,\sigma_K(\rho(K,u)u)\rangle\, \HH^{d-1}(\D u).
\end{align*}

Now, we will use the map $T:S^{d-1}\to \partial K$, $u\to \rho(K,u)u$, which is bijective and bilipschitz. From Lemma~2.4 in \cite{Hug1996b}, the Jacobian of $T$ is
$$JT(u)=\frac{\|\nabla h_{K^*}(u)\|}{h_{K^*}(u)^d}$$
for $\HH^{d-1}$ almost all $u\in S^{d-1}$.
Using this fact, we get that
\begin{align*}
&\lim_{n\to\infty} n^{\frac{2}{d+1}}\,\E_{\mu_q}(V_\lambda(K^{(n)}\cap K_1)-V_\lambda(K))\\
&\quad=c_d\int_{\partial K}\rho(K, x/\|x\|)^{d+1}\,\tilde{q}(\nabla h_{K^*}(x/\|x\|))^{-\frac{2}{d+1}}\,\lambda(\rho(K,x/\|x\|),x/\|x\|) \\
&\quad\quad\times \|\nabla h_{K^*}(x/\|x\|)\| \, \kappa(x)^{-\frac{1}{d+1}}\langle x/\|x\|,\sigma_K(x)\rangle\, h_{K^*}(x/\|x\|)^d \,\HH^{d-1}(\D x)\\
&\quad=c_d\int_{\partial K}\tilde{q}(\nabla h_{K^*}(x))^{-\frac{2}{d+1}}\,\lambda(\|x\|,x/\|x\|)\,\|\nabla h_{K^*}(x)\|\,\kappa(x)^{-\frac{1}{d+1}} \\
&\quad\quad\times \langle x,\sigma_K(x)\rangle \, \HH^{d-1}(\D x)\\
&\quad=c_d\int_{\partial K} q(\|\nabla h_{K^*}(x)\|^{-1}, \nabla h_{K^*}(x)/\|\nabla h_{K^*}(x)\|)^{-\frac{2}{d+1}}\,\lambda(\|x\|,x/\|x\|) \\
&\quad\quad \times \|\nabla h_{K^*}(x)\| \, \kappa(x)^{-\frac{1}{d+1}}\langle x,\sigma_K(x)\rangle \, \HH^{d-1}(\D x)\\
&\quad=c_d\int_{\partial K}q(h_K(\sigma_K(x)), \sigma_K(x))^{-\frac{2}{d+1}}\,\lambda(\|x\|,x/\|x\|)\,\kappa(x)^{-\frac{1}{d+1}}\,\HH^{d-1}(\D x),
\end{align*}
which completes the proof.
\end{proof}

\subsection{Upper bound on the variance for circumscribed polytopes}
In this subsection, we prove Theorem~\ref{thm:var-up-circ}. In fact, we again prove more than strictly necessary for the verification of Theorem~\ref{thm:var-up-circ}. The main content of the subsection is described in the following theorem. 
 
\begin{theorem} \label{thm_var-up-vol}
Let $K$ be a convex body in $\mathbb{R}^d$ with $o \in \inti K$ which slides freely in a ball. 
Assume that the function $q: [0,\infty) \times S^{d-1} \rightarrow [0,\infty)$ satisfies the properties i), ii) and iii). 
Let $\lambda: \R \times S^{d-1} \rightarrow [0,\infty)$ be a locally integrable function which is continuous at each point of the set $\{(\rho(K,u),u):u\in S^{d-1}\}$.
Then
\[ \Var_{\mu_q}(V_\lambda(K^{(n)}\cap K_1)) \ll n^{-\frac{d+3}{d+1}}, \]
where the implied constant depends only on $K$, $\lambda$ and $q$.
\end{theorem}

It is clear that Theorem~\ref{thm_var-up-vol} directly implies Theorem~\ref{thm:var-up-circ}.

\begin{proof}
The Efron-Stein jackknife inequality yields
\begin{align*}
&\Var_{\mu_q}(V_\lambda(K^{(n)}\cap K_1)) \\
& \ \leq (n+1)\, \E_{\mu_q}\left(V_\lambda(K^{(n)}\cap K_1) - V_\lambda(K^{(n+1)}\cap K_1)\right)^2 \\
& \ \ll n \, \E_{\mu_q} \left(V_\lambda(K^{(n)}\cap K_1) - V_\lambda(K^{(n+1)}\cap K_1)\right)^2.
\end{align*}

We can now use the same technique as in the proof of Theorem ~\ref{main}.
Since $\PP_{\mu_q}(K^{(n)}\not\subset K_1)\ll \alpha^n$ for some real number $\alpha\in (0,1)$ depending on the convex body $K$ and the density $q$ (cf. \cite{BoSch2010}) and since the distributions of the random polyhedral sets $K^{(n)}$ and $(K^*_{(n)})^*$ are the same, it follows that
\begin{align*}
&\Var_{\mu_q}(V_\lambda(K^{(n)}\cap K_1)) \\
& \ \ll n \left( \E_{\mu_q} \left(\mathbf{1}\{K^{(n)} \subset K_1\} (V_\lambda(K^{(n)}) - V_\lambda(K^{(n+1)}))^2 \right) + O(\alpha^n) \right) \\
& \ = n \left( \E_{\varrho, K^*} \left(\mathbf{1}\{(K_{(n)}^*)^* \subset K_1\} (V_\lambda((K_{(n)}^*)^*) - V_\lambda((K_{(n+1)}^*)^*))^2 \right) + O(\alpha^n) \right) \\
& \ = n \left( \E_{\varrho, K^*} \left(\int_{S^{d-1}} \int_{h_{[x_1,\dots,x_n]}(u)}^{h_{[x_1,\dots,x_{n+1}]}(u)} \widetilde{\lambda}(s,u) \, \D s \,\mathcal{H}^{d-1}(\D u) \right)^2 + O(\alpha^n) \right) \\
& \ \ll n \left( \E_{\varrho, K^*} \left( W_{\widetilde{\lambda}}(K_{(n+1)}^*) - W_{\widetilde{\lambda}}(K_{(n)}^*) \right)^2 + O(\alpha^n) \right),
\end{align*}
where $x_1,\dots,x_{n+1} \in K^* \backslash K^*_1$ and $\widetilde{\lambda}$ is defined as in the proof of Theorem~\ref{thm:vol-mu}.

From Proposition~\ref{propmod} it follows that a ball rolls freely in $K^*$. 
Thus, Theorem~\ref{thm:var-up_q} yields the upper bound on the variance.
\end{proof}

Finally, we remark that the proof of Theorem~\ref{thm:lln-circ} is analogous to the one of Theorem~\ref{thm:law-of-large-num} using the monotonicity of $V_\lambda(K^{(n)}\cap K_1)-V_\lambda(K)$  with respect to $n$.

\section{Appendix}
In this section, we describe some results that are used in the proofs of 
Theorems~\ref{thm:vol-mu} and \ref{thm_var-up-vol}. 
The arguments are taken from \cite{Hug2000}. Let $B^d(t,R)$ denote a Euclidean ball 
with center $t$ and radius $R$.
 
\begin{lemma}
\label{LemmaA} 
Let $R>0$, $t\in \R^d$ and $\|t\|<R$.  
Then $B^d(t,R)^{*}$ is an ellipsoid of revolution.  
The lengths of its semiaxes are
\[a_{1}=\frac{R}{R^{2}-\|t\|^{2}}\quad\mbox{\rm and}\quad a_{2}=\cdots=a_{d}=\frac{1}{\sqrt{R^{2}-\|t\|^{2}}}.\] 
\end{lemma}

\begin{proof}  
We may assume that $t=\|t\|e_{1}$, where $(e_{1},\ldots,e_{d})$ is the standard orthonormal basis of $\R^d$. 
Then, for $u\in S^{d-1}$, we have
\[\rho(B^d(t,R)^*,u)=h_{B^d(t,R)}(u)^{-1}=(\langle t,u\rangle +R)^{-1}.\] 
If we set $x:=\rho(B^d(t,R)^{*},u)u$, $u\in S^{d-1}$, then we find  
\[\|x\|^{2}=\frac{1}{R^{2}}\left(1-\frac{\langle t,u\rangle}{\langle t,u\rangle +R}\right)^{2}=\frac{1}{R^{2}}(1-\langle t,x\rangle)^{2}.\] 
Hence, any boundary point $x$ of $B^d(t,R)^{*}$ satisfies the equation
\[x_{1}^{2}-\left(\frac{\|t\|}{R}\right)^{2}x_{1}^{2}+\frac{2\|t\|}{R^{2}}x_{1} +\sum_{i=2}^{d}x_{i}^{2}=\frac{1}{R^{2}},\]
where $x_{1},\ldots,x_{d}$ are the coordinates of $x$ with respect to  $e_{1},\ldots,e_{d}$, the standard basis of $\R^d$.
Elementary calculations finally show that this is equivalent to
\[\frac{\left(x_{1}+\frac{\|t\|}{R^{2}-\|t\|^{2}}\right)^{2}}{\left(\frac{R}{R^{2}-\|t\|^{2}}\right)^{2}}+\sum_{i=2}^{d}\frac{x_{i}^{2}}{\left(\frac{1} {\sqrt{R^{2}-\|t\|^{2}}}\right)^{2}}=1.\] 
This proves that $\partial B^d(t,R)^{*}$ is contained in and thus coincides with the boundary of an ellipsoid with semiaxes as described in the statement of the lemma.
\end{proof}

\begin{lemma}
\label{LemmaB} 
Let $R>0$ be fixed.  
Let $\mathcal{E}_{t}$, $t\in [0,R)$, be an ellipsoid of revolution in $\R^d$ the semiaxes of which have the lengths $a_{1}=R\omega^{2}$ and $a_{2}=\ldots=a_{d}=\omega$, where
$\omega:=(R^{2}-t^{2})^{-\frac{1}{2}}$.  
Then the principal radii of curvature of $\mathcal{E}_{t}$ can be uniformly bounded from below by $R^{-1}$.
\end{lemma} 

\begin{proof}
In order to obtain the principal radii of curvature of $\mathcal{E}_{t}$, we consider a general ellipsoid $\mathcal{E}(a,b)$ of revolution with semiaxes lengths $a_{1}=a$ and $a_{2}=\ldots= a_{d}=b$, where $a,b>0$. 
By choosing suitable coordinates, we obtain that $\mathcal{E}(a,b)=A\cdot B^d$, where $A=(a_{ij})$ is a real $d \times d$ matrix with $a_{ii}=a_{i}$, for $i\in \{1,\ldots,d\}$, and $a_{ij}=0$, for $i\neq j$ and $i,j\in \{1,\ldots,d\}$. 
Then elementary calculations yield for the second derivatives of the support function $h:=h_{\mathcal{E}(a,b)}(\cdot)$ at $x\in \R^d\setminus\{o\}$ that
\[\frac{\partial^{2} h}{\partial x_{i}\partial x_{j}}(x)=\frac{a^{2}b^{2}}{\|Ax\|^{3}}\cdot\left\{\begin{array}{ll}
\displaystyle{\sum_{k=2}^{d}x_{k}^{2}}\,,&i=j=1\,,\\
\displaystyle{x_{1}^{2}+\frac{b^2}{a^2}\sum_{k=2,\, k\neq
i}^{d}x_{k}^{2}\,,}&i=j\neq 1\,,\\
\displaystyle{-x_{1}x_{j}}\,,&i=1\neq j\,,\\[1ex]
\displaystyle{-\frac{b^{2}}{a^2}x_{i}x_{j}}\,,&1\neq i\neq
j\neq 1\,.  \end{array}\right.\]
In order to determine the principal radii of curvature of ${\mathcal{E}}(a,b)$, we can restrict ourselves to the case where $x=(x_{1},x_{2},0,\ldots,0)$ and $\|x\|=1$, since ${\mathcal{E}}(a,b)$ has rotational symmetry.  
In addition, the eigenvectors $u_{2},\ldots,u_{d}\in S^{d-1}$ of the reverse Weingarten map $\overline{W}_{x}$ of ${\mathcal{E}}(a,b)$ at $x$ are equal to the eigenvectors of the Weingarten map of ${\mathcal{E}}(a,b)$ at the uniquely determined boundary point of ${\mathcal{E}}(a,b)$ with exterior unit normal vector $x$ (compare \cite{Sch2014}, \S 2.5).  
The latter are given by
\[u_{2}=(x_{2},-x_{1},0,\ldots,0)\quad\mbox{\rm and}\quad u_{i}=e_{i}\,,\quad
i\in \{3,\ldots,d\}\,;\] see, e.g., Chapter 3, IV in \cite{Spivak1979}.  
Finally, Lemma 2.5.1 in \cite{Sch2014} and some further calculations yield for the corresponding principal radii of curvature $r_{2}(x),\ldots,r_{d}(x)$ of ${\mathcal{E}}(a,b)$ at $x$ that
\[r_{2}(x)=\overline{\rm II}_{x}(u_{2},u_{2})=\frac{a^{2}b^{2}}{\sqrt{a^{2}x_{1}^{2}+b^{2}x_{2}^{2}}^{3}} \] 
and
\[r_{i}(x)=\overline{\rm II}_{x}(u_{i},u_{i})=\frac{b^{2}}{\sqrt{a^{2}x_{1}^{2}+b^{2}x_{2}^{2}}}\,, \quad i\in \{3,\ldots,d\},\]
where $\overline{\rm II}_x$ denotes the reverse second fundamental form of $\mathcal{E}(a,b)$ at $x$.
  
From these calculations we obtain for  $r_{2}(x),\ldots,r_{d}(x)$, the principal radii of curvature of the ellipsoid $\mathcal{E}_{t}$ in direction $x=(x_{1},x_{2},0,\ldots,0)\in S^{d-1}$, that
\[r_{2}(x)=\overline{\rm II}_{x}(u_{2},u_{2})=R^{-1}\left(1-\left(\frac{t}{R}\right)^{2}+\left(\frac{t}{R}\right)^{2}x_{1}^{2}\right)^{-\frac{3}{2}}\] 
and
\[r_{i}(x)=\overline{\rm II}_{x}(u_{i},u_{i})=R^{-1}\left(1-\left(\frac{t}{R}\right)^{2}+\left(\frac{t}{R}\right)^{2}x_{1}^{2}\right)^{-\frac{1}{2}},\quad i\in\{3,\ldots,d\}.\]
This proves the lemma, since $x_{1}^{2}\in [0,1]$.
\end{proof}
 
\begin{prop}
\label{propmod} 
Let $K$ be a convex body in $\R^d$ with $o \in \inti K$ and assume that the polar body $K^*$ is a summand of the ball $B^d(o,R)$ with $R>0$.  
Then $B^d(o,R^{-1})$ rolls freely inside $K$.
\end{prop}

\begin{proof}
By Theorem 3.2.2 in \cite{Sch2014}, the assumption implies that $K^*$ slides freely inside $B^d(o,R)$. 
Hence for each $u\in S^{d-1}$ there is a (uniquely determined) point  $x^{*}=\tau_{K^{*}}(u)\in \partial K^{*}$ with $K^{*}\subset B^d(x^*-Ru,R)$, where $\tau_{K^{*}}$ denotes the reverse spherical image of $K^*$.
Thus we get $B^d(x^*-Ru,R)^*\subset K$.  
In particular, the support set of $K^*$ at $u$ is equal to $ \{x^*\}$ and $h_{K^*}(u)=h_{B^d(x^*-Ru,R)}(u)$.  
Therefore, $h_{K^*}(u)^{-1}u\in \partial K$ and also $h_{K^*}(u)^{-1}u\in \partial (B^d(x^*-Ru,R)^*)$.  
This shows that \[h_{K^*}(u)^{-1}u\in \partial K \cap \partial(B^d(x^*-Ru,R)^*).\] 
In other words, for each $x\in \partial K$ there is some $u\in S^{d-1}$ such that 
\[x\in B^d(\tau_{K^{*}}(u)-Ru,R)^{*}\subset K.\] 
From Lemma \ref{LemmaA} we know that $B^d(\tau_{K^{*}}(u)-Ru,R)^{*}$ is an ellipsoid of revolution.  
Since $o\in {\rm int}\, K^{*}$ and $K^{*}\subset B^d(\tau_{K^{*}}(u)-Ru,R)$, we obtain $|\tau_{K^{*}}(u)-Ru|<R$. 
Now Lemma \ref{LemmaA}, Lemma \ref{LemmaB} and a special case of Corollary 3.2.13
 in \cite{Sch2014} imply that $B^d(o,R^{-1})$ rolls freely inside any of the
ellipsoids $B^d(\tau_{K^{*}}(u)-Ru,R)^{*}$.  
But then $B^d(o,R^{-1})$ rolls freely inside $K$.
\end{proof}

\section{Acknowledgements}
This research of the first author was supported by the J\'anos Bolyai Research Scholarship of the Hungarian Academy of Sciences.

Supported by the European Union and co-funded by the European Social Fund
under the project ``Telemedicine-focused research activities on the field of Mathematics,
Informatics and Medical sciences'' of project number ``T\'AMOP-4.2.2.A-11/1/KONV-2012-0073''.

The second and the third author were partially supported by the German research foundation (DFG) under the grant HU1874/4-2.


\begin{bibdiv}
\begin{biblist}

\bib{Ba2008}{article}{
   author={B{\'a}r{\'a}ny, Imre},
   title={Random points and lattice points in convex bodies},
   journal={Bull. Amer. Math. Soc. (N.S.)},
   volume={45},
   date={2008},
   number={3},
   pages={339--365},
   % issn={0273-0979},
   %review={\MR{2402946 (2009e:52009)}},
   %doi={10.1090/S0273-0979-08-01210-X},
}

\bib{BoFoHu2010}{article}{
   author={B{\"o}r{\"o}czky, K{\'a}roly~J.},
   author={Fodor, Ferenc},
   author={Hug, Daniel},
   title={The mean width of random polytopes circumscribed around a convex
   body},
   journal={J. Lond. Math. Soc. (2)},
   volume={81},
   date={2010},
   number={2},
   pages={499--523},
   % issn={0024-6107},
   % review={\MR{2603007 (2011b:52007)}},
   %doi={10.1112/jlms/jdp077},
}
 
\bib{BoFoReVi2009}{article}{
   author={B{\"o}r{\"o}czky, K{\'a}roly~J.},
   author={Fodor, Ferenc},
   author={Reitzner, Matthias},
   author={V{\'{\i}}gh, Viktor},
   title={Mean width of random polytopes in a reasonably smooth convex body},
   journal={J. Multivariate Anal.},
   volume={100},
   date={2009},
   number={10},
   pages={2287--2295},
   % issn={0047-259X},
   % review={\MR{2560369 (2010j:52011)}},
   %doi={10.1016/j.jmva.2009.07.003},
}

\bib{BoSch2010}{article}{
   author={B{\"o}r{\"o}czky, K{\'a}roly J.},
   author={Schneider, Rolf},
   title={The mean width of circumscribed random polytopes},
   journal={Canad. Math. Bull.},
   volume={53},
   date={2010},
   number={4},
   pages={614--628},
   issn={0008-4395},
}

\bib{CaSchYu2013}{article}{
   author={Calka, Pierre},
   author={Schreiber, Tomasz},
   author={Yukich, J.~E.},
   title={Brownian limits, local limits and variance asymptotics for convex
   hulls in the ball},
   journal={Ann. Probab.},
   volume={41},
   date={2013},
   number={1},
   pages={50--108},
   % issn={0091-1798},
   %review={\MR{3059193}},
   %doi={10.1214/11-AOP707},
}

\bib{CaYu2014}{article}{
   author={Calka, Pierre},
   author={Yukich, J.~E.},
   title={Variance asymptotics for random polytopes in smooth convex bodies},
   journal={Probab. Theory Related Fields},
   volume={158},
   date={2014},
   number={1-2},
   pages={435--463},
   % issn={0178-8051},
   %review={\MR{3152787}},
   %doi={10.1007/s00440-013-0484-1},
}

\bib{Federer1969}{book}{
  title = {Geometric measure theory},
  publisher = {Springer},
  date = {1969},
  author = {Federer, Herbert},
  place = {Berlin}
}

%\bib{Gardner2006}{book}{
 % title = {Geometric tomography},
 % publisher = {Cambridge University Press},
 % year = {2006},
 % author = {Gardner, Richard~J.},
 % series = {Encyclopedia of Mathematics and its Applications 58},
 % address = {Cambridge},
 % edition = {2. ed.},
%}

\bib{Gr2007}{book}{
   author={Gruber, Peter~M.},
   title={Convex and discrete geometry},
   series={Grundlehren der Mathematischen Wissenschaften [Fundamental
   Principles of Mathematical Sciences]},
   volume={336},
   publisher={Springer},
   place={Berlin},
   date={2007},
   pages={xiv+578},
   % isbn={978-3-540-71132-2},
   % review={\MR{2335496 (2008f:52001)}},
}

% \bib{Hug1994}{thesis}{
  % author = {Hug, Daniel},
  % title = {Geometrische Ma{\ss}e in der affinen Konvexgeometrie},
  % type = {Dissertation, Univ. Freiburg},
  % date = {1994}
% }

\bib{Hug1996}{article}{
  author = {Hug, Daniel},
  title = {Contributions to Affine Surface Area},
  journal = {Manuscripta Mathematica},
  date = {1996},
  volume = {91},
  pages = {283--301},
  number = {1},
  publisher = {Springer},
}


\bib{Hug1996b}{article}{
   author={Hug, Daniel},
   title={Curvature relations and affine surface area for a general convex
   body and its polar},
   journal={Results Math.},
   volume={29},
   date={1996},
   number={3-4},
   pages={233--248},
   % issn={0378-6218},
}


\bib{Hug1998}{article}{
   author={Hug, Daniel},
   title={Absolute continuity for curvature measures of convex sets I},
   journal={Math. Nachr.},
   volume={195},
   date={1998},
   %number={3-4},
   pages={139--158},
   % issn={0378-6218},
}



\bib{Hug2000}{thesis}{
  author={Hug, Daniel},
  title={Measures, curvatures and currents in convex geometry}, 
  type={Habilitationsschrift, Univ. Freiburg}, 
  date={2000},
}

\bib{Hu2013}{article}{
   author={Hug, Daniel},
   title={Random polytopes},
   conference={
      title={Stochastic geometry, spatial statistics and random fields},
   },
   book={
      series={Lecture Notes in Math.},
      volume={2068},
      publisher={Springer, Heidelberg},
   },
   date={2013},
   pages={205--238},
   %review={\MR{3059649}},
   %doi={10.1007/978-3-642-33305-7_7},
}

\bib{HugSchn2013}{article}{
  author = {Hug, Daniel},
  author = {Schneider, Rolf},
  title = {H\"older continuity of normal cycles and of support measures of convex
	bodies},
  year = {submitted},
  pages = {arXiv:1310.1514v1}
}


% \bib{Last2014}{article}{
  % author = {Last, G\"unter},
  % author = {Peccati, Giovanni}, 
  % author = {Schulte, Matthias},
  % title = {Normal approximation on Poisson spaces: Mehler's formula, second
	% order Poincar\'e inequalities and stabilization},
  % journal = {submitted},
  % year = {2014}
% }


% \bib{Leichtweiss1986}{article}{
  % author = {Leichtwei{\ss}, Kurt},
  % title = {Zur Affinoberfl\"ache konvexer K\"orper},
  % journal = {Manuscripta Mathematica},
  % year = {1986},
  % volume = {56},
  % pages = {429--464},
  % number = {4},
  % publisher = {Springer}
% }

% \bib{L1998}{book}{
   % author={Leichtwei{\ss}, Kurt},
   % title={Affine geometry of convex bodies},
   % publisher={Johann Ambrosius Barth Verlag, Heidelberg},
   % date={1998},
   % pages={x+310},
   % isbn={3-335-00514-7},
   % review={\MR{1630116 (2000j:52005)}},
% }

\bib{Reitzner2003}{article}{
   author={Reitzner, Matthias},
   title={Random polytopes and the Efron-Stein jackknife inequality},
   journal={Ann. Probab.},
   volume={31},
   date={2003},
   number={4},
   pages={2136--2166},
   % issn={0091-1798},
   %review={\MR{2016615 (2005b:60026)}},
   %doi={10.1214/aop/1068646381},
}


% \bib{ReSu1968}{article}{
   % author={R{\'e}nyi, A.},
   % author={Sulanke, R.},
   % title={Zuf\"allige konvexe Polygone in einem Ringgebiet},
   % language={German},
   % journal={Z. Wahrscheinlichkeitstheorie und Verw. Gebiete},
   % volume={9},
   % date={1968},
   % pages={146--157},
   %%review={\MR{0229272 (37 \#4846)}},
% }

\bib{Sch2014}{book}{
  title = {Convex bodies: the Brunn-Minkowski theory},
  publisher = {Cambridge University Press},
  date = {2014},
  author = {Schneider, Rolf},
  series = {Encyclopedia of Mathematics and its Applications 151},
  place = {Cambridge},
  edition = {2. exp. ed.},
  isbn = {978-1-107-60101-7},
  size = {XXII, 736 S.},
}

\bib{Spivak1979}{book}{
  author = {Spivak, Michael},
  title = {A comprehensive introduction to differential geometry},
  volume = {3},
  place = {Berkeley},
  publisher = {Publish or Perish},
  date = {1979},
  edition = {2. ed.},
  % isbn = {0-914098-82-9},
  size = {XI, 466 S. : graph. Darst.}
}

% \bib{Turk2010}{thesis}{
  % author={T\"urk, Ines},
  % title={Zuf\"allige Polytope und polyedrische Mengen}, 
  % type={Diplomarbeit, Karlsruhe Institute of Technology}, 
  % date={2010},
% }

% \bib{Vitale1992}{article}{
  % author = {Vitale, Richard A.},
  % title = {On the Bias of the Jackknife Estimate of Variance},
  % journal = {Lecture Notes-Monograph Series},
  % date = {1992},
  % volume = {22},
  % pages = {399--403},
  % publisher = {Institute of Mathematical Statistics}
% }


\bib{WeWie1993}{article}{
   author={Weil, Wolfgang},
   author={Wieacker, John A.},
   title={Stochastic geometry},
   conference={
      title={Handbook of convex geometry, Vol.\ A, B},
   },
   book={
      publisher={North-Holland, Amsterdam},
   },
   date={1993},
   pages={1391--1438},
   % review={\MR{1243013 (95a:60014)}},
}

\end{biblist}
 
\end{bibdiv}
\end{document}